\documentclass[12pt,a4paper]{amsart}  \makeatletter
\renewcommand\normalsize{%
    \@setfontsize\normalsize{11.7}{14pt plus .3pt minus .3pt}%
    \abovedisplayskip 10\p@ \@plus4\p@ \@minus4\p@
    \abovedisplayshortskip 6\p@ \@plus2\p@
    \belowdisplayshortskip 6\p@ \@plus2\p@
    \belowdisplayskip \abovedisplayskip}
\renewcommand\small{%
    \@setfontsize\small{9.5}{12\p@ plus .2\p@ minus .2\p@}%
    \abovedisplayskip 8.5\p@ \@plus4\p@ \@minus1\p@
    \belowdisplayskip \abovedisplayskip
    \abovedisplayshortskip \abovedisplayskip
    \belowdisplayshortskip \abovedisplayskip}
\renewcommand\footnotesize{%
    \@setfontsize\footnotesize{8.5}{9.25\p@ plus .1pt minus .1pt}
    \abovedisplayskip 6\p@ \@plus4\p@ \@minus1\p@
    \belowdisplayskip \abovedisplayskip
    \abovedisplayshortskip \abovedisplayskip
    \belowdisplayshortskip \abovedisplayskip}
\ifdefined\pdfpagewidth
\else
\fi
\calclayout
\makeatother

\usepackage{graphicx,calc}
\usepackage{amsmath,amsfonts,amssymb,amsthm}
\usepackage{pstricks} 
\usepackage{epstopdf}
\usepackage{srcltx}
\usepackage{comment}
\usepackage[normalem]{ulem}
\usepackage{hyperref}
\usepackage[all]{xy}
\usepackage[utf8]{inputenc} 
\usepackage[T1]{fontenc} 
 \usepackage[abs]{overpic}
\usepackage{pinlabel}

\newtheorem{theorem}{Theorem}[section]

\newtheorem{proposition}[theorem]{Proposition}
\newtheorem{lemma}[theorem]{Lemma}
\newtheorem{corollary}[theorem] {Corollary}

\theoremstyle{remark}
\newtheorem{remark}[theorem]{Remark}

\theoremstyle{definition} 
\newtheorem{definition}{Definition} [section]

\usepackage{aeguill}                 
\usepackage{graphicx}
\usepackage{amsmath,amsfonts,amssymb,amsthm, color}
\usepackage{pstricks} 
\usepackage{epstopdf}
\usepackage{srcltx}
\usepackage[utf8]{inputenc} 
\usepackage{hyperref} 
\usepackage[english]{babel} 
\usepackage{comment}

\newcommand{\Area}{\operatorname{Area}}
\newcommand{\dehn}{\operatorname{Dehn}}

\def\<{\langle}
\def\>{\rangle} 
\def\Z{\mathbb{Z}}

\def\N{\mathbb{N}}
\def\g{\gamma}
\def\im{{\rm{im}\, }}
\def\CL{{\rm{CL}}}
\def\Ann{{\rm{Ann}}}
\def\Z{\mathbb{Z}}
\def\E{\mathbb{E}}

\def\e{\epsilon}
\def\A{\mathbb{T}}
\def\wt{\widetilde}

\def\bpq{B_{pq}}

\def\bpqp{B_{pq}^+}
\def\tbpqp{\tilde{B}_{pq}^+}

\def\onto{\twoheadrightarrow}

\newcommand{\dist}{{\rm{dist}}}
\newcommand{\ssm}{\smallsetminus}

\title[Snowflake groups and conjugator length]{Snowflake groups and conjugator length functions with non-integer exponents} 
\author{M.\ R.\ Bridson and T.\ R.\ Riley}

\date{15 Dec 2025}
 
\usepackage[all]{xy}
\usepackage{bbm}

\begin{document}

\begin{abstract} We exhibit novel geometric phenomena in the study of conjugacy problems for discrete groups.
We prove that the snowflake groups $\bpq$, indexed by pairs of positive integers
$p>q$, have conjugator length functions $\CL(n)\simeq n$ and annular Dehn functions $\Ann(n) \simeq n^{2\alpha}$, where $\alpha = \log_2(2p/q)$.  Then, building on $\bpq$, we construct groups $\tbpqp$, for which  $\CL(n)\simeq n^{\alpha+1}$.  Thus the conjugator length spectrum and the spectrum of exponents of annular Dehn functions are both dense in the range $[2,\infty)$.   
\end{abstract}

\thanks{The first author thanks the Mathematics Department of Stanford University for its hospitality during
his sabbatical in the Spring of 2025.  The second author gratefully acknowledges the financial support of the National Science Foundation (NSF GCR-2428489). ORCID: 0000-0002-0080-9059 (MRB)  and 0009-0004-3699-0322 (TRR)}  

\def\IP{{\bf{IP}}}
\def\bCL{{\bf{CL}}}

\def\th{\theta}

\renewcommand{\theenumi}{\roman{enumi}}

\maketitle 
\section{Introduction} 
Conjugator length functions provide bounds of a geometric nature on the difficulty
of conjugacy problems in finitely generated  groups,  and annular Dehn functions provide an
alternative bound in the case of finitely presented groups.
The purpose of this paper is to 
show that these functions  exhibit a wide variety of behaviours. 
Concentrating on finitely presented groups and
functions of the form $n\mapsto n^\alpha$, we shall prove that in both cases the
{\em spectrum of exponents} $\alpha$ that arise is dense in the range  $[2,\infty)$.  This parallels a foundational result concerning   {\em Dehn functions},  a class of functions that have been intensively studied in connection with the
word problem for finitely presented groups.   By definition,  the Dehn function of a finitely presented group $G=\<A\mid R\>$ is
$$\dehn_G(n) = \max\lbrace \Area(u) \mid |u| \le n,\ u =1 \text{ in } G \rbrace,$$  where $u$ is a  word in the letters $A^{\pm 1}$
 that represents $1\in G$ and $\Area(u)$ is  the least integer $N$ such that $u$ is equal in the free group $F(A)$
to a product of $N$ conjugates of defining relations $r\in R^{\pm 1}$.  In more geometric language, 
$N$ is the combinatorial area (i.e.~number of 2-cells) in a minimal van~Kampen diagram for $u$.
 The {\em isoperimetric spectrum} $\IP$ is the countable set of numbers $e \ge 1$ such that there is a finitely presented group whose Dehn function is $\simeq n^e$,   where $\simeq$ is the standard equivalence relation of geometric 
 group theory (see Section~\ref{sec:conventions}).   Gromov showed that $\IP$ contains a gap between $1$ and $2$; see \cite{Gromov, Ol2, Bowditch, Papasoglu5}. Brady and Bridson \cite{BB} showed that this is the only gap; in other words, the
closure of $\IP$ is $\{1\}\cup [2,\infty)$.

The {\em conjugator length function} $\CL: \N \to \N$ of a finitely generated group $G$ is defined by
$$\CL_G(n) = \max \lbrace \CL(u,v) \mid |u| + |v| \le n,\ u\sim v \rbrace$$  where $u$ and $v$ are words in the generators and $\CL(u,v)$ is the length of a shortest element conjugating $u$ to $v$ in $G$; up to $\simeq$ equivalence,  this is independent of the choice of
generating set. 
The set of numbers $e$ such that  $n^e\simeq \CL_G(n)$ for some finitely presented group $G$  is a countable
subset of $[1,\infty)$, which we call the {\em{conjugator length spectrum}} $\bCL$.    In \cite{BrRi1, BrRi2} we proved that $\N\subset\bCL$, but no integer exponents were known before the present work.  The non-integer exponents in the following theorem are  transcendental.

\begin{theorem} \label{CL thm}
For every pair of positive integers $p>q$ there exists a finitely presented group $G$ with $\CL_G(n)\simeq n^{\alpha+1}$,
where $\alpha = \log_2(2p/q)$.
\end{theorem}

\begin{corollary} $\bCL$ is dense in the range $[2,\infty)$. 
\end{corollary} 

\noindent We do not know whether there are gaps in $\bCL\cap [1,2]$.
\bigskip

The \emph{annular Dehn function} of a finitely presented group $G$, defined by Brick and Corson in \cite{BC} and   recently revisited by Gillis and Riley in \cite{GiRi}, is the function $\Ann_G : \N \to \N$ defined by  $$\Ann_G(n) = \max \lbrace \Ann(u,v) \mid |u| + |v| \le n,\ u\sim v \rbrace$$  where $u$ and $v$ are words in the generators and $\Ann(u,v)$ is the minimal $N$ such that there exists a word $w$ such that $uw=wv$ in $G$ and $\Area_G( w^{-1}uwv^{-1}) = N$ or, equivalently, there is an \emph{annular diagram} that exhibits the conjugacy $u \sim v$ and has combinatorial area $N$.    As observed in \cite{BC}, 
for any finitely presented group $G$,
\begin{equation} \label{ann by area and cl} 
\dehn_G(n) \ \leq \   \Ann_G(n) \ \leq \  \dehn_G(2\CL(n)+n).
\end{equation}
The first inequality here comes from specializing to conjugacies $u \sim v$ where $v$ is the empty word and the second 
inequality holds because if $|u| + |v| =n$ and $u \sim v$, then there is a $w$ such that $w^{-1}uw v^{-1}$ has length at most $2\CL(n)+n$ and represents the identity in $G$.   

\begin{theorem} \label{Ann thm}
For every pair of positive integers $p>q$ there exists a finitely presented group $G$ with $\Ann_G(n)\simeq n^{2\alpha}$, where $\alpha = \log_2(2p/q)$.    
\end{theorem}

This theorem shows that ${\textbf{Ann}}$ is dense in the range $(2,\infty)$ and contains all integers greater than $2$.
The observation that $ \Ann_{\Z^2}(n)\simeq n^2$ adds the exponent $2$.  If $\Ann_G(n)\lesssim n^e$ for $e<2$,
then the first inequality in (\ref{ann by area and cl}) tells us that $\dehn_G(n) \lesssim n^e$,  so $\dehn_G(n) \simeq n$ and $G$ is hyperbolic \cite{Gromov4}.
For hyperbolic groups,  $\CL_G(n)\simeq n$ and  therefore $ \Ann_G(n)\simeq n$, by the second inequality in (\ref{ann by area and cl}). 

\begin{corollary} 
$\mathbb{N}\subset \textbf{\textup{Ann}}$ and the closure of  
$\textbf{\textup{Ann}}$ is $\{1\}\cup [2,\infty)$.  
\end{corollary}

\bigskip

The constructions that we use to prove these results start with the  {\em snowflake groups} 
that Brady and Bridson used to prove that the closure of $\IP$ is $\{1\}\cup [2,\infty)$:
 \begin{equation}\label{bpq}
\bpq = \left\langle a, b, s, t \mid [a,b]=1,\, s^{-1}a^qs=a^pb,\ t^{-1}a^qt=a^pb^{-1} \right\rangle.
 \end{equation}
 If $p>q$ then $\dehn_{\bpq}(n)\simeq n^{2\alpha}$, where $\alpha=\log_2(2p/q)$.   This is a family of {\em tubular groups}: $\bpq$ is the fundamental group of the 2-complex obtained
 from the torus with fundamental group   $\<a,b\>\cong\Z^2$ by attaching two cylinders, with one end of 
 each cylinder wrapping $q$ times around the loop $a$ and the other ends wrapping around the loops $a^pb$
 and $a^pb^{-1}$, respectively.  
 
 The conjugator length function of $\bpq$ is not exotic,  in fact  it is linear (Theorem~\ref{t:cl}).
 This mundane conclusion may seem disappointing, 
but  in the light of (\ref{ann by area and cl}) it leads immediately to a proof of  Theorem \ref{Ann thm}.
\begin{theorem}\label{t:main1} 
If $p>q$ then $\Ann_{\bpq}(n)\simeq n^{2\alpha}$,  
where  $\alpha = \log_2(2p/q)$.
\end{theorem}

Our search for more interesting conjugator length functions continues with the family of 
tubular groups obtained by adding an extra cylinder to the 2-complex of $\bpq$,
 this time with both ends wrapping around the loop $b$.  The fundamental group of this
 complex is the 
 HNN extension of $\bpq$ obtained by adding a new stable letter $\theta$ that commutes with $b$:
 \begin{equation}\label{e:bpqp}
\bpqp = \left\langle a, b, s, t,\th \mid [a,b]=1, \,  [b,\th]=1,\, s^{-1}a^qs=a^pb,\ t^{-1}a^qt=a^pb^{-1}\right\rangle.
 \end{equation} 
 These groups also have linear conjugator length functions (Theorem~\ref{t:cl}), but something more exotic happens when
 we take the following  central
extension of $\bpqp$:
 \begin{equation}\label{e:tbpqp}
\tbpqp = \left\langle a, b, s, t , \th, z  \ \left| \ \parbox{55mm}{$z$ central, $[b,\th]=z$,  $[a,b]=1$, \\ $s^{-1}a^qs =a^pb$, $t^{-1}a^qt=a^pb^{-1}$} \right. \right\rangle.
 \end{equation}  
 
One can regard  $\tbpqp $ as an HNN extension of $\bpq \times \Z$ with $\Z=\< z\>$, where the stable letter
of this extension is $\theta$, conjugating $\<b,z\>$ to itself via the isomorphism $[z\mapsto z, \ b\mapsto bz]$.
Alternatively,  noting that $H=\< b, \th, z\>< \tbpqp$ is a copy of the 3-dimensional Heisenberg group,  one can regard $\tbpqp $
as the amalgamated free product of $\bpq \times \Z$ with $H$,  amalgamating the two copies of $\<b,z\>$ by
the isomorphism implicit in the notation. In either description, it is clear that the centre $\<z\>$ is infinite.

\begin{theorem}\label{t:main}  If $p>q$ then $\CL_{\tbpqp}(n)\simeq n^{\alpha +1}$,  where  $\alpha = \log_2(2p/q)$.
\end{theorem} 

Theorem~\ref{CL thm}   follows.  We will also prove (in Sections~\ref{s:df of bpqp}, \ref{s:skew slabs}, and  \ref{s:dist of centre in tbpqp}) the following auxiliary results.      

\begin{theorem} \label{t:aux} 
For all positive integers $p>q$,  writing $\alpha = \log_2(2p/q)$, we have:
\begin{enumerate}
\item \label{aux:i} The Dehn function of $\bpqp$ is $\simeq n^{2\alpha}$.
\item \label{aux:ii} \label{aux:iii} $\CL_{\bpq}(n) \simeq  \CL_{\bpqp}(n)\simeq n$. 
\item \item  \label{aux:iv} The distortion of   $Z=\<z\><\tbpqp$ is  $\dist_{Z}(n) \simeq n^{\alpha+1}$.  
\end{enumerate}
\end{theorem} 
 
\subsection*{Outline}  
In Section~\ref{sec:conventions}  we fix notation and some basic facts about van~Kampen (disc) diagrams and their conjugacy-problem analogues,  annular diagrams.  Following a brief discussion 
of the Dehn function of $\bpqp$ in Section~\ref{s:df of bpqp},   we embark on the main business of this paper,  the proofs of Theorems~\ref{t:main} and \ref{t:aux}.

The proof in  \cite{BB}  that the Dehn function of $\bpq$ is $n^{2\alpha}$  relied on the study of a
certain set of  ``snowflake"  words $w_k$ with fractal properties (Definition~\ref{Snowflake words} here).
These words were used to prove that the  distortion  of the
free abelian subgroup $\A=\<a,b\> <\bpq$ is $\simeq n^\alpha$.  Our proof
of Theorem \ref{t:main}  relies heavily on a more refined version of this
last result,  which compels us to revisit the geometry of snowflake words and geodesics in $\bpq$---this is done in Section~\ref{sec:Snowflake words}.

In Section~\ref{s:skew slabs}  we show that the conjugacy length functions of $\bpq$ and $\bpqp$ are linear. 
We do so by means of a geometric argument that applies to any multiple-HNN extension of a free-abelian  groups in which the amalgamated subgroups are cyclic and form a family that is 
suitably \emph{skew} (Theorem~\ref{t:cl}).

In Section~\ref{s:dist of centre in tbpqp} we establish  the lower bound $\CL_{\tbpqp} (n) \succeq n^{\alpha +1}$ (Lemma~\ref{l: lower bound on CL}).  We observe first that the shortest word conjugating $b$ to $bz^{M}$   in $\tbpqp$ is $\theta^M$, because that is so in  the Heisenberg group  $\langle b, \theta, z \mid z \text{ central}, \  [\theta, b]= z \rangle$ onto which  $\tbpqp$ retracts.  The second key point is that $d_{\tbpqp}(1, bz^{M}) \approx n$ when $M \approx  n^{\alpha +1}$,
which is a reflection of  the distortion of the center $\langle z \rangle$ in $\tbpqp$.  This
distortion  can be estimated by counting cells labelled $[b, \th]$ in van Kampen diagrams over $\bpqp$, which in turn we can bound thanks to our distortion estimates in $\bpq$.

The remainder of the paper is devoted to proving that $\CL_{\tbpqp} (n) \preceq n^{\alpha +1}$.  
We focus first, in Sections~\ref{s: special conjugacies} and \ref{s: centralizer and zeta-maps}, on the special case of conjugate elements $g \sim g z^N$ where $N \in \Z$.  Because  $\tbpqp$ is a central extension of $\bpqp$ by $\langle z \rangle \cong \Z$,  the set of integers $N$ such that  $g \sim g z^N$ in  $\tbpqp$  is the image of a homomorphism $\zeta_g$ from the centralizer $C_{\bpqp}(g)$ to  $\Z$, which we call a \emph{zeta-map}.  We determine (in Proposition~\ref{p:zetas})   what $C_{\bpqp}(g)$ and $\zeta_g$ are case-by-case in terms of $g$. Given generators for $C_{\bpqp}(g)$,  some $\Z$-linear combination of their images $m_1, \ldots, m_r$ under $\zeta_g$ equals $N$,  in other words
the linear diophantine equation $a_1 m_1 + \cdots + a_r m_t =N$ has a solution $a_1, \ldots, a_r \in \Z$.  We argue (in Lemma~\ref{l:bezout}) that by considering alternative integer solutions one can gain control on the absolute values of the 
integers $a_i$, and this gives us the desired control  on the length of conjugating elements.
 
In Section~\ref{s:Proof of Main Theorem}, where we bound the conjugator length for general $u \sim v$ in $\tbpqp$,  we will employ the following strategy: instead of  conjugating $u$ to $v$ directly,  we take their images  $\bar{u} \sim \bar{v}$ in $\bpqp$ and, calling on the fact that $\CL_{\bpqp}(n) \simeq n$ (Theorem~\ref{t:cl}), we conjugate each to a preferred representative $\-{u}_0$ of their conjugacy class using  conjugators whose lengths we can bound.  Lifting these
conjugacies to $\tbpqp$, we get a conjugacy $u \sim vz^N$ where the discrepancy   $z^N$  lies in the center.  But then $u_0 \sim u_0 z^N$ and we can correct for the discrepancy by deploying a further conjugating word supplied by the special case analysed in Sections~\ref{s: special conjugacies} and \ref{s: centralizer and zeta-maps}

In order to execute  this strategy,  we need to make our choice of $\-{u}_0$ carefully,  controlling not only its
length but also its centralizer and its zeta-map $\zeta_{\-{u}_0}$; this choice is made in Section~\ref{sec: preferred reps}.
Our efforts to bound  the lengths of the conjugators that arise at this stage of the proof rely on somewhat fine
properties of the metric on the subgroup $\langle a, b, \th \rangle$ in $\bpqp$. We prove the necessary estimates in Section~\ref{s:word metric in bpqp}, building on our understanding of  the geometry of $\langle a, b \rangle$ in $\bpq$,  as described in  Proposition~\ref{p:distort}.

\begin{remark}[Other families of exponents]
 We prove in  
 Theorem~\ref{CL thm} that for every pair of integers $p > q >0$ and for $\alpha = \log_2(2p/q)$, there is a finitely presented group with $\CL(n) \simeq n^{1 + \alpha}$.   We shall explain in a subsequent article that by taking a central extension of $\bpq$ and amalgamating its centre with the 3-dimensional Heisenberg group,  one can construct
  a finitely presented group with $\CL(n) \simeq n^{2 \alpha}$.   This provides a different set of exponents that is dense  in $(2, \infty)$.   More generally,  by taking a central extension of $\bpq$ and amalgamating its centre with that of the  class-2 nilpotent group $G_m$ constructed in \cite{BrRi1},  we get finitely presented groups with $\CL(n) \simeq n^{m-1 + 2 \alpha}$.   
\end{remark}

\section{Preliminaries} \label{sec:conventions}

Given functions $f,g:\N \to \N$, we write $f \lesssim g$ when there  is a constant $C>0$ such that $f(n)  \leq   Cg(Cn+n)+Cn+C$ for all $n$,  and  $f \simeq g$ if $f\lesssim g$ and $g\lesssim f$. 
Up to this standard equivalence relation,  $\CL_G(n)$ does not depend on the choice of finite generating set for the group $G$ and  $\Ann_G(n)$ and $\dehn_G(n)$  do not depend on the choice of finite presentation.

We use $=$ to indicate when two words represent the same element in a group and $\equiv$ to indicate when two words are  letter-by-letter the same.  We write $[x,y]$ for the commutator $x^{-1} y^{-1} xy$.

Suppose $G$ is a group given by a finite presentation $\langle A \mid R \rangle$.
A \emph{van Kampen diagram} for a word $w$ (in the generators $A$) is a finite, planar,  contractible 2-complex,
with a base vertex on the boundary and edges that are directed and labeled by elements of $A$ in a manner which
ensures that the boundary cycle of each 2-cell is labeled by a relation from $R^{\pm 1}$, and the boundary 
cycle of the diagram read from the basepoint is labeled   $w$. 

The importance of  van Kampen diagrams  is summarized in \emph{van Kampen's Lemma}, which states that a word $w$ represents the identity element in $G$ if and only if it admits a van Kampen diagram. The \emph{area} of a word $w$, denoted $\Area(w)$, is the minimum number of faces (2-cells) in any van Kampen diagram for $w$. 

An \emph{annular diagram} for a pair of words $u$ and $v$ is a finite, planar, combinatorial 2-complex that is homotopy equivalent to a circle or, in the degenerate case, a point. Its edges are directed and labeled by generators from $A$. Around each face, one reads a relation from $R^{\pm 1}$. The words $u$ and $v$ label the two boundary cycles of the diagram, read
with clockwise orientation from some vertex in the cycle.

The  analogue of van Kampen's Lemma for annular diagrams is as follows: 
if there is an annular diagram for  $u$ and $v$,  then $u$ and $v$ are conjugate in $G$  (we write $u \sim v$);
conversely if $u \sim v$ and $u\neq 1$ in $G$, then there is an annular diagram for  $u$ and $v$.

Our recent article  \cite{BrRiSa} with Andrew~Sale provides a survey of conjugator length functions and a careful treatment of these preliminaries.

\section{The Dehn function of $\bpqp$ } \label{s:df of bpqp}

We will give careful proofs of claims (2)--(4) of Theorem~\ref{t:aux},  but since (1) is not directly relevant to 
our study of conjugator length functions,  we only sketch its proof.  The  bound
$n^{2\alpha}\preceq  \delta_{\bpqp}(n)$ is an immediate consequence of the 
facts that killing $\th$ retracts $\bpqp$ onto $\bpq$ and $n^{2\alpha}\simeq  \delta_{\bpq}(n)$.
The complementary bound $ \delta_{\bpqp}(n)\preceq n^{2\alpha} $ can be established by
following the proof of Proposition 3.2 in \cite{BB}. The translation requires only superficial
adjustments and changes to constants,  once one has the  estimate on the distortion of $\<b\>$ in $\bpq$ (and hence
in $\bpqp$) coming from Proposition~\ref{p:distort}.

\section{Snowflake words and the distortion of $\A=\<a,b\>$ in $B=B_{pq}$} \label{sec:Snowflake words}

We fix the positive integers $p>q$  and throughout this section we write $B$ in place of $B_{pq}$,
to avoid a clutter of notation.  We retain the notation  $\alpha = \log_2 (2p/q)$,  and we note that $\alpha >1$.

\begin{definition}[Snowflake words] \label{Snowflake words}
For each positive integer $N$ we
want to define a {\em snowflake word}  $w_N$ in the letters $\{a,b^{\pm 1}, s^{\pm 1}, t^{\pm 1}\}$ so 
that $w_N=a^N$ in $B$ and $w_N$ is relatively short (as quantified in Lemma \ref{l:snowflake-words}).
These words are
constructed by a simple recursion,  the geometry of which is illustrated in Figure~\ref{fig:triangle}.
To make the early stages of the recursion clean,  we define
$w_N\equiv a^N$ if $N< 2p$.  Then,  for $N\ge 2p$, having defined $w_n$ for $n<N$, we  
write $N = 2pN_0 + \e_0$ with $0\le \e_0 < 2p$ and define
$$
w_N \equiv a^{\e_0} (s^{-1} w_{qN_0}s)  (t^{-1} w_{qN_0} t).
$$
By induction,  $w_{qN_0} = a^{qN_0}$ in $B$, so $s^{-1} w_{qN_0}s = (s^{-1}a^qs)^{N_0} = (a^pb)^{N_0}$ and
$t^{-1} w_{qN_0}t = (t^{-1}a^qt)^{N_0} = (a^pb^{-1})^{N_0}$. Therefore
$$
w_N = a^{\e_0} (a^pb)^{N_0}(a^pb^{-1})^{N_0} = a^{2pN_0+\e_0} = a^N, 
$$
as required.
\end{definition}

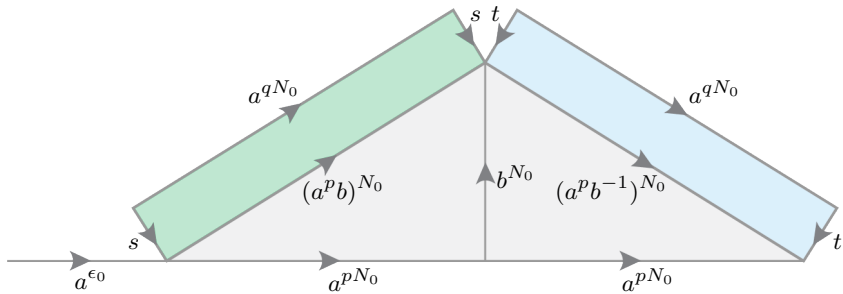
\begin{figure}[ht]
\begin{overpic}
{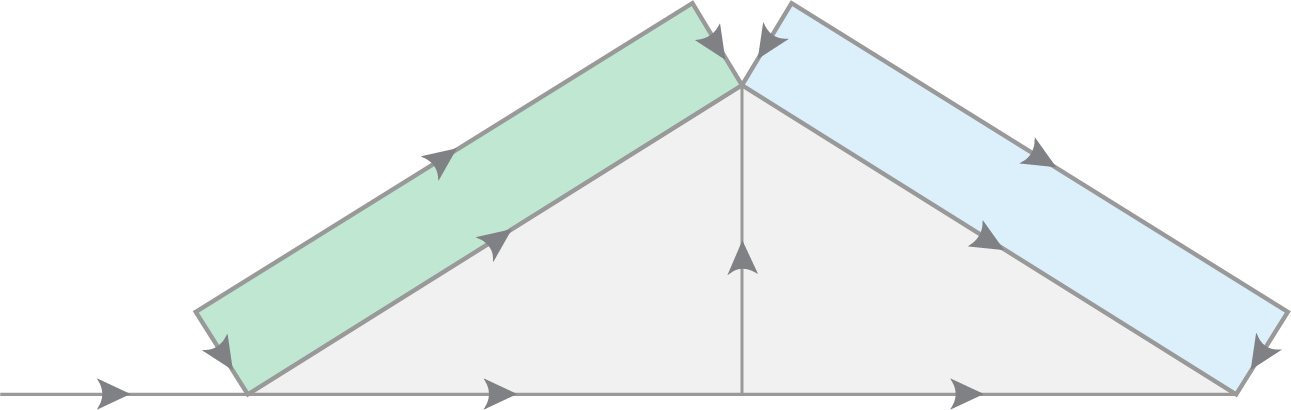}
 \put(45,8){\small{$s$}}     
 \put(173,94){\small{$s$}}    
  \put(309,8){\small{$t$}}      
 \put(181,94){\small{$t$}}     
 \put(25,-7){\small{$a^{\e_0}$}}     
 \put(90,63){\small{$a^{qN_0}$}}     
 \put(183,31){\small{$b^{N_0}$}}     
 \put(255,63){\small{$a^{qN_0}$}}     
 \put(110,27){\small{$(a^pb)^{N_0}$}}   
  \put(205,27){\small{$(a^pb^{-1})^{N_0}$}}   
  \put(120,-8){\small{$a^{pN_0}$}}     
 \put(230,-8){\small{$a^{pN_0}$}}   
\end{overpic}
 \caption{The recursion for constructing snowflake words}
  \label{fig:triangle}
\end{figure}

\begin{lemma} \label{l:snowflake-words} 
For all positive integers $d$,  if $N\le (2p/q)^d$ and $C\ge 2p+3$ then $|w_N| \le C(2^d-1)$.
\end{lemma}

\begin{proof} With $d$ and $C$ fixed,   we proceed by induction on $N$.  If  $N\le 2p$, then $w_N\equiv a^N$, and so $|w_N|= N \le 2p$, which is less than $C(2^d-1)$ for all $d\ge 1$.  For the inductive step, 
we recall that if
\begin{equation}\label{N_0}
N = 2pN_0 + \e_0
\end{equation}
 with $0\le \e_0 < 2p$, then
\[
w_N \equiv a^{\e_0} (s^{-1} w_{qN_0}s)  (t^{-1} w_{qN_0} t)
\]
and therefore
\begin{equation}\label{trian}
 |w_N| \le \e_0 + 4 + 2 |w_{qN_0}| < 2p + 4 + 2 |w_{qN_0}|.
\end{equation} 
From (\ref{N_0}) we have   $qN_0 \le (q/2p) N$, so $N\le (2p/q)^d$ implies $qN_0 \le (2p/q)^{d-1}$.
Therefore,  by induction, we have from (\ref{trian}) 
\[ |w_N| < 2p + 4 + 2C  (2^{d-1}-1) = C (2^d -1) + (2p+4 - C).
\]
As $C\ge 2p+3$, this completes the induction.
\end{proof}

\begin{lemma}\label{l:distort} Let $r=2p/q$ and let $\alpha = \log_2 r$.
There are positive constants $k_i, K_i \ (i=0,1,2)$ so that, for all integers $N>0$,
\begin{enumerate}
\item  $k_0 N^{1/\alpha} \le d_B(1, a^N) \le K_0 N^{1/\alpha}$
\item $k_1 N^{1/\alpha} \le d_B(1, (a^pb)^N) \le K_1 N^{1/\alpha}$
\item  $k_2 N^{1/\alpha} \le d_B(1, (a^pb^{-1})^N) \le K_2 N^{1/\alpha}$.
\end{enumerate} 
\end{lemma}

\begin{proof}  As $a^pb = s^{-1}a^qs$ and  $a^pb^{-1} = t^{-1}a^qt$, items (ii) and (iii) follow easily from (i), so we
concentrate on proving (i).  First we establish the upper bound. 
 
Given $N$,  let $d>0$ be the integer such that $r^{d-1}\le N < r^{d}$. Then
\begin{equation}\label{u1}
 d-1\le \log_r N = \frac{\log_2 N}{\log_2 r} = \frac{1}{\alpha} \log_2N = \log_2 N^{1/\alpha}.
\end{equation} 
From Lemma \ref{l:snowflake-words}, for the snowflake word $w_N$, taking $C= 2p+3$ we have
\begin{equation}\label{u2}
d_B(1, a^N) \le |w_N| \le C (2^d-1) < C2^d = (4p+6) 2^{d-1}.
\end{equation}
Combining (\ref{u1}) and (\ref{u2}), we conclude that 
$$d_B(1, a^N) \le (4p+6) 2^{\log_2 N^{1/\alpha}} = (4p+6) N^{1/\alpha},$$
so it suffices to let $K_1= 4p+6$. 

The existence of the constants $k_0, k_1, k_2$ is the content of \cite[Proposition~2.1]{BB}. 
\end{proof}

The following proposition sharpens the main subgroup-distortion result in \cite{BB}. The argument is peculiar to $B$ most acutely in that it relies on the abelian nature of the vertex group $\mathbb{T}$, which allows us to  rearrange   $v_1 \cdots v_m$ using a permutation $\sigma$.  

\def\k{\kappa}

\begin{proposition}\label{p:distort}
There exists a constant  $\k >1$ such that, for all $g\in \A=\<a,b\>$, in $B=B_{pq}$ we have
\[ \frac{1}{\k} \,  d_\A(1, g)^{1/\alpha} \le d_B(1, g) \le \k\,  d_\A(1, g)^{1/\alpha}.
\]
\end{proposition}

\begin{proof} It is enough to prove that there exist positive constants $k,K$
such that 
\[k \,  d_\A(1, g)^{1/\alpha} \le d_B(1, g) \le K\,  d_\A(1, g)^{1/\alpha},
\]
and this is what we shall do. One can then define $\k = \max \{K,  1/k\}$. 

The existence of the constants $k$ and $K$, but not their values, is independent of the choice of 
finite generating sets for $B$ and $\A$. For the rightmost inequality, it is convenient to work with the
generating set $\{a, a^pb\}$ for $\A$.   Then,  given $g=a^N(a^pb)^M$,   we have $d_\A(1, g) = |M| + |N|$.
On the other hand,  using the triangle inequality and Lemma \ref{l:distort} we have
$$
d_B(1, g) \le d_B(1,a^N) + d_B(1,(a^pb)^M) \le K_0 |N|^{1/\alpha} + K_1 |M|^{1/\alpha},
$$
so by Minkowski's inequality, 
$$
d_B(1, g)^\alpha \le K^\alpha (|N| + |M|) = K^\alpha \, d_\A(1, g),
$$
where $K= \max\{K_0, K_1\}$.    Raising both sides to the power $1/\alpha$, 
we obtain the desired upper bound on  $d_B(1, g)$. 

In order to establish the complementary lower bound on $d_B(1, g)$, we consider the nature of geodesic representatives
for $g$ in $B$, working with the standard 
generating set $\{a, b, s, t\}$.  The following argument can be made purely algebraic,  but it
is more instructive to consider  van Kampen diagrams.

\begin{figure}[ht]
\begin{overpic}
{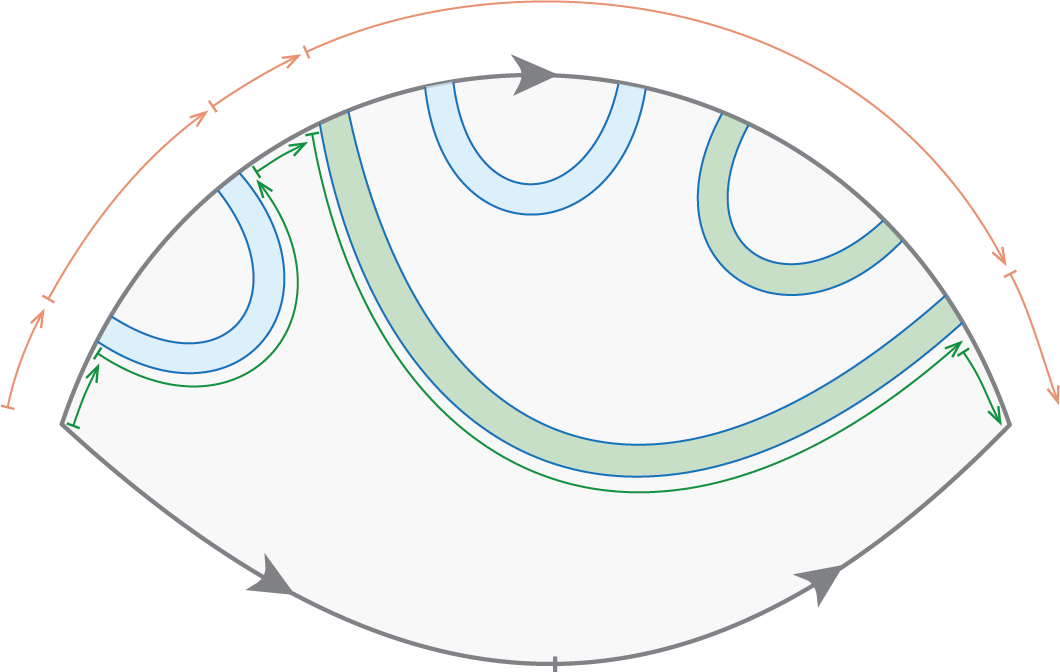}
 \put(127,150){\small{$\omega$}}     
 \put(56,8){\small{$a^N$}}    
 \put(200,9){\small{$b^M$}}    
 \put(230,140){\small{$\Delta$}}    
 \put(100,30){\small{$\Delta_0$}}    
  \put(42,90){\small{$\Delta_1$}}      
 \put(139,85){\small{$\Delta_2$}}     
 \put(125,130){\small{$\Delta_3$}}     
 \put(190,110){\small{$\Delta_4$}} 
 \put(22,60){\small{$u_1$}}    
  \put(60,66){\small{$u_2$}}      
 \put(67,116){\small{$u_3$}}     
 \put(108,48){\small{$u_4$}}     
 \put(227,60){\small{$u_5$}}      
 \put(-5,75){\small{$v_1$}}    
  \put(19,119){\small{$v_2$}}      
 \put(53,147){\small{$v_3$}}     
 \put(185,154){\small{$v_4$}}     
 \put(252,82){\small{$v_5$}}  
  \put(78,137){\small{$s$}}      
  \put(178,136){\small{$s$}}      
  \put(217,108){\small{$s$}}      
  \put(232,88){\small{$s$}}      
  \put(20,83){\small{$t$}}      
  \put(51,121){\small{$t$}}      
  \put(104,144){\small{$t$}}      
  \put(153,144){\small{$t$}}      
\end{overpic}
 \caption{Illustrating our proof of Proposition~\ref{p:distort}}
  \label{fig:diagram}
\end{figure}
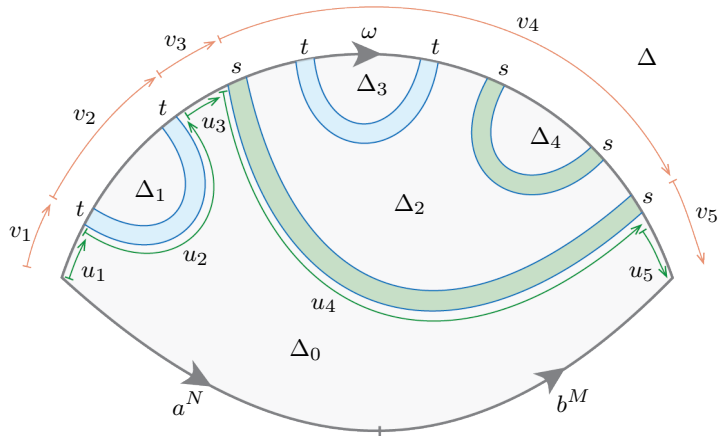

Let $g=a^Nb^M$, let $\omega$ be
a  geodesic representative of $g$, and consider a minimal-area van Kampen diagram with boundary label $\omega(a^Nb^M)^{-1}$ as illustrated in Figure~\ref{fig:diagram}.
All $s$-  and $t$-corridors in $\Delta$ have their endpoints on the arc of $\partial\Delta$ labelled $\omega$.
By deleting the interior of all corridors, we cut $\Delta$ into a disjoint union of contractible subdiagrams
 $\Delta_0, \ldots, \Delta_l$, one of which $\Delta_0$
contains the original boundary arc labelled $a^Nb^M$.   The label on the boundary cycle of $\Delta_0$ has the form $u_1u_2\dots u_m  (a^Nb^M)^{-1}$, where each  $u_i$ is a word in the free group on $\{a,b\}$
that is either the label on the side of an $s$-corridor  or $t$-corridor, or else is the label on a portion of the boundary cycle
$\partial\Delta$.  The words $u_i$ labelling the sides of corridors are of three types: a power of $a^q$ or of $a^pb$
or of $a^pb^{-1}$.  We regard the remaining $u_i$, those coming from $\partial\Delta$,  as being of a fourth type.
Let $v_i$ be the (geodesic) subword of $\omega$ labelling the arc connecting the endpoints 
of $u_i$ and note that since $u_i=v_i$ lies in the abelian subgroup $\A=\<a,b\>$,  for every permutation $\sigma \in 
{\rm{sym}}(m)$
$$
v_{\sigma (1)} \cdots v_{\sigma (m)}
$$  is also a geodesic representative of
$g$ in $B$.  

We apply  a permutation $\sigma$ that brings all of the  $u_i$ of each of the four types together, so as to obtain a new word in the letters $a,b$ representing $g$
\[ g=u_0a^{qN_1}(a^pb)^{N_2}(a^pb^{-1})^{N_3}\]
where 
\[ d_B(1,g) = |\omega| = |u_0| + d_B(1, a^{qN_1}) + d_B(1, (a^pb)^{N_2}) + d_B(1, (a^pb^{-1})^{N_3}).
\] 
Using Lemma \ref{l:distort}  and taking $k_3:=\min \{1, k_0, k_1, k_2\}$, we get 
\begin{equation}\label{ugh1}
\frac{1}{k_3}\ d_B(1,g) \ge |u_0| +   (qN_1)^{1/\alpha} + N_2^{1/\alpha} + N_3^{1/\alpha}.
\end{equation}
On the other hand,  the triangle inequality in $\A$ tells us that 
\[
d_\A(1,g) \le  |u_0| + q |N_1| + (p+1)|N_2| + (p+1)|N_3|  \leq    (p+1) ( |u_0|  + q|N_1| + |N_2| + |N_3|).
\]  
Therefore,  
\begin{equation}\label{ugh3}
\frac{1}{4(p+1)}\ d_\A(1,g) \leq  \max \lbrace  |u_0|,\  q |N_1|,\   |N_2|,\  |N_3|\rbrace =:\mu.
\end{equation}
From  (\ref{ugh1}) we have $d_B(1,g) \ge k_3 \mu^{1/\alpha}$, so from    (\ref{ugh3}) we conclude that
\[ d_B(1,g) \ge \frac{k_3}{(4(p+1))^{1/\alpha}}\ d_\A(1,g)^{1/\alpha}.
\]
Setting $k = k_3/(4(p+1))^{1/\alpha}$ completes the proof.
\end{proof}
 
\section{Skew slabs, annular diagrams, and conjugacy in $\bpq$}  \label{s:skew slabs}

Suppose $A \cong\Z^r$ is a  free abelian group with word metric $d$.

\begin{definition}
A pair $L, L'$ of   cyclic subgroups in $A$ is  {\em skew}   if  each is infinite and
$L \cap L'=1$. A family $L_1, \dots,L_m$ of cyclic subgroups of $A$ is {\em skew} 
if every  pair $L_i, L_j$ such that $L_i \neq L_j$ is skew.
\end{definition}

\begin{lemma}\label{l:skew}  
For every skew pair $L, L'$ of  cyclic subgroups  in $A$, there exists a constant $\eta>0$
such that if $x, y \in A$,  $\ell \in L$ and $\ell' \in L'$ satisfy $x \ell = \ell'y$, then
$$
 \max\{ d(1, \ell),\, d(1, \ell')\}\le \eta\,   \max \{ d(1,x),\, d(1,y) \}.
$$ 
\end{lemma}

\begin{proof}
This follows immediately from the fact that skew lines in Euclidean space $\E^r$ diverge linearly and
the standard embedding $\Z^r\hookrightarrow\E^r$  is bi-Lipschitz 
when $\Z^r$ is endowed with the word metric $d$.
\end{proof}

We will use this lemma to prove:

\begin{theorem}\label{t:cl}
$\CL_{\bpq}(n) \simeq\CL_{\bpqp}(n) \simeq n$.
\end{theorem}

In fact, we prove the following more general theorem which applies both to $G = \bpq$ and to $G=\bpqp$.  
In the following statement it is implicitly assumed that $A$ is the group generated by $a_1,\dots,a_r$ and that
 $c_i$ and $c_i'$ are words in these generators.

\begin{theorem}\label{t:cl2} 
Consider an HNN extension  $G$ of  $A = \langle a_1, \ldots, a_r \rangle \cong \Z^r$ with $m$ stable letters,
  $$G = \left\langle \, a_1, \ldots, a_r, t_1, \ldots, t_m \mid [a_j, a_{j'}] =1, \ 
  \ t_i^{-1} c_i t_i = c'_i  \  \ \forall i,j, j'\, \right\rangle,$$
 where  $t_i$ conjugates the infinite cyclic subgroup $L_i = \langle c_i \rangle$ to  $L'_i = \langle c'_i \rangle$ and 
the family $L_1, L'_1, \ldots, L_m, L'_m \leq A$ is skew.  
Suppose that there are constants $\alpha>1$ and $k>1$ such that for all $x \in A$,
$$
\frac{1}{k} d_G(1,x)^{1/\alpha} \le d_A(1,x) \le k\, d_G(1,x)^{1/\alpha}.
$$
Then $\CL_G(n)\simeq n$.
\end{theorem}

\begin{proof} By hypothesis, no cyclic subgroup of $A$ is exponentially distorted in $G$.  So, if powers $x^{\lambda}$ and $x^{\mu}$ of a non-trivial element $x$ of $A$ are conjugate in $G$, then  $\lambda = \pm \mu$.

Consider a minimal-area annular diagram $\Delta$   that portrays a conjugacy from $v$ to $u$, with
$v$ as the label in the inner boundary cycle and $u$ as the label on the outer boundary cycle.
We consider the  $t_i$-corridors and annuli in this diagram.  
General considerations (detailed in \cite{BrRiSa}) tell us that we can assume there are no inessential $t_i$-annuli in $\Delta$.     

Suppose there is an essential  $t_i$-annulus.  We will argue that $\Delta$ has at most two $t_i$-annuli for each $i$ and that
and if there are two then the $t_i$-edges in one annulus point towards the inner boundary cycle and the $t_i$-edges
in the other annulus point outwards.  To see that this is the case, note that 
if there are two  $t_i$-annuli, then their boundary cycles are labelled by $c_i^{\lambda}$, ${c'}_i^{\lambda}$,  $c_i^{\mu}$, and ${c'}_i^{\mu}$ for some $\lambda, \mu \in \Z$ and these four words represent conjugate elements in $G$, so as above, 
$\lambda = \pm \mu$.  If it were the case
that $\lambda = \mu$,  we  could delete from $\Delta$ either the  interior 
of the subdiagram that the cycles labelled $c_i^{\lambda}$  and  $c_i^{\mu}$ cobound or the subdiagram
that the cycles labelled ${c'}_i^{\lambda}$  and  ${c'}_i^{\mu}$ cobound, and  then identify  the cycles; at least one of these two operations would reduce area, contrary to our assumption that $\Delta$ has minimal area. This analysis excludes the possibility of there being three essential $t_i$-annuli because any three must include a pair that gives rise to the reduction we have just described (but it leaves open the possibility two $t_i$-annuli provided $\lambda = - \mu$).   
 
We have argued that there can be at most $2m$ 
essential $t_i$-annuli in $\Delta$. The next thing to observe is that these annuli must form a single stack, 
as illustrated on the left in  Figure~\ref{fig:annuli},. To see why this is true,  note if there were a non-empty subdiagram between 
two successive corridors in the radial order, then this subdiagram would illustrate a conjugacy in $A$ between
the label on the outer boundary cycle of one $t_i$-annulus and the label on the inner boundary boundary cycle 
of the next annulus; as $A$ is abelian, these labels would have to be equal, as group elements (read from any basepoint); 
in fact,  since each label is a power of $c_i$ or $c_i'$ and the family of amalgamated cyclic subgroups
 is skew (which excludes the possibilities $c_i^{n_1}= c_j^{n_2}$ and $c_i^{n_1}= {c'}_j^{n_2}$ with  $|n_1|\neq |n_2|$),
 the labels must be equal (read from suitable basepoints),  and 
we could delete the interior of the subdiagram to obtain a smaller-area diagram
contradicting the assumed minimality of $\Delta$. 

 Suppose that the stack of $t_i$-annuli 
 is non-empty and  consider the two complementary subdiagrams (which may be degenerate).  If there are no $t_i$-corridors in
 one of these,  the inner one $\Delta_0$ say, then it portrays a conjugacy in $A$ from $v$ to the word
 labelling the inner boundary cycle of the stack of corridors. 
 As $A$ is abelian,  this latter word (which is a power
 of some $c_i$ or $c_i'$)
 is equal to $v$ in $A$.  (Note that the absence of $t_i$-corridors emanating from
 the cycle labeled $v$ forces $v$ to be a word in the generators   $a_j$ of $A$.) 
 We now modify $\Delta$ by replacing  $\Delta_0$ with a van~Kampen
 diagram portraying this equality.  The basepoint of this van Kampen diagram is the point of the boundary at which the
 label $v$ begins, and this basepoint also lies on the inner boundary cycle of the stack of annuli.  
 
 This argument shows that if there are no $t_i$-corridors in either of the two subdiagrams of $\Delta$ complementary
 to the stack of annuli,  then $u$ and $v$ are words  in the generators of  $A$ and after modification we can assume that
 there is a path in the 1-skeleton of $\Delta$ from the outer boundary cycle of $\Delta$ to the inner boundary cycle, crossing each  $t_i$-annulus exactly once, and so $u$ and $v$ have cyclic permutations that can be conjugated to each other by some word $W = w_0 t_{i_1}^{\varepsilon_1} w_1 \cdots t_{i_k}^{\varepsilon_k} w_k$ where $w_0, \ldots, w_k$ are words on the generators of $A$,  $k \leq 2m$, $i_1, \ldots, i_k \in \{1, \ldots, m\}$, and $\varepsilon_1, \ldots, \varepsilon_k \in \{1, -1\}$.   
 For every prefix $W_0$ of $W$, we have $W_0^{-1} u W_0 \in A$, and so, because $A$ is abelian,    $T = t_{i_1}^{\varepsilon_1}  \cdots t_{i_k}^{\varepsilon_k}$  conjugates $u$ to $v$.
 The length of $T$ is  $k \leq 2m$, so in this case we have $\CL(u,v) \le 2m + (|u|+|v|)/2$.
 
 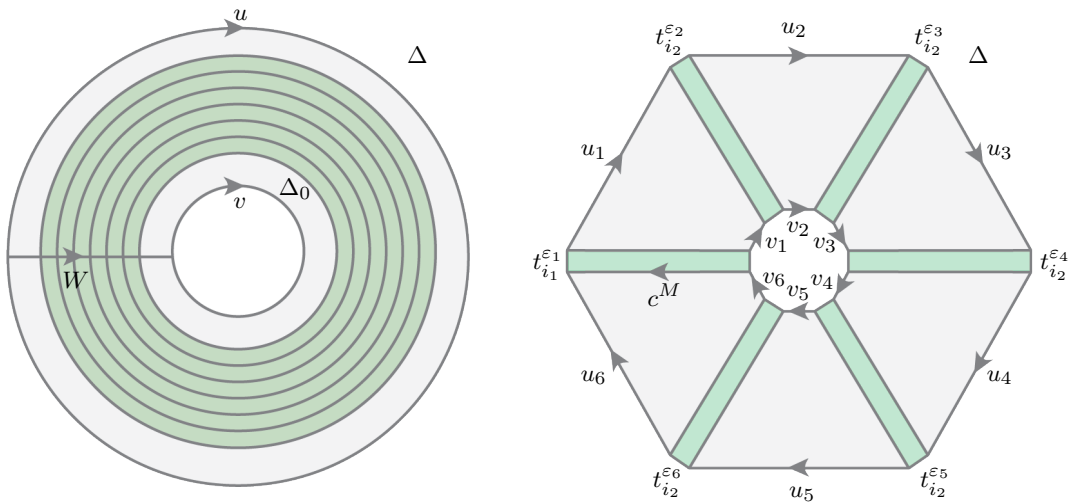
\begin{figure}[ht]
\begin{overpic}
{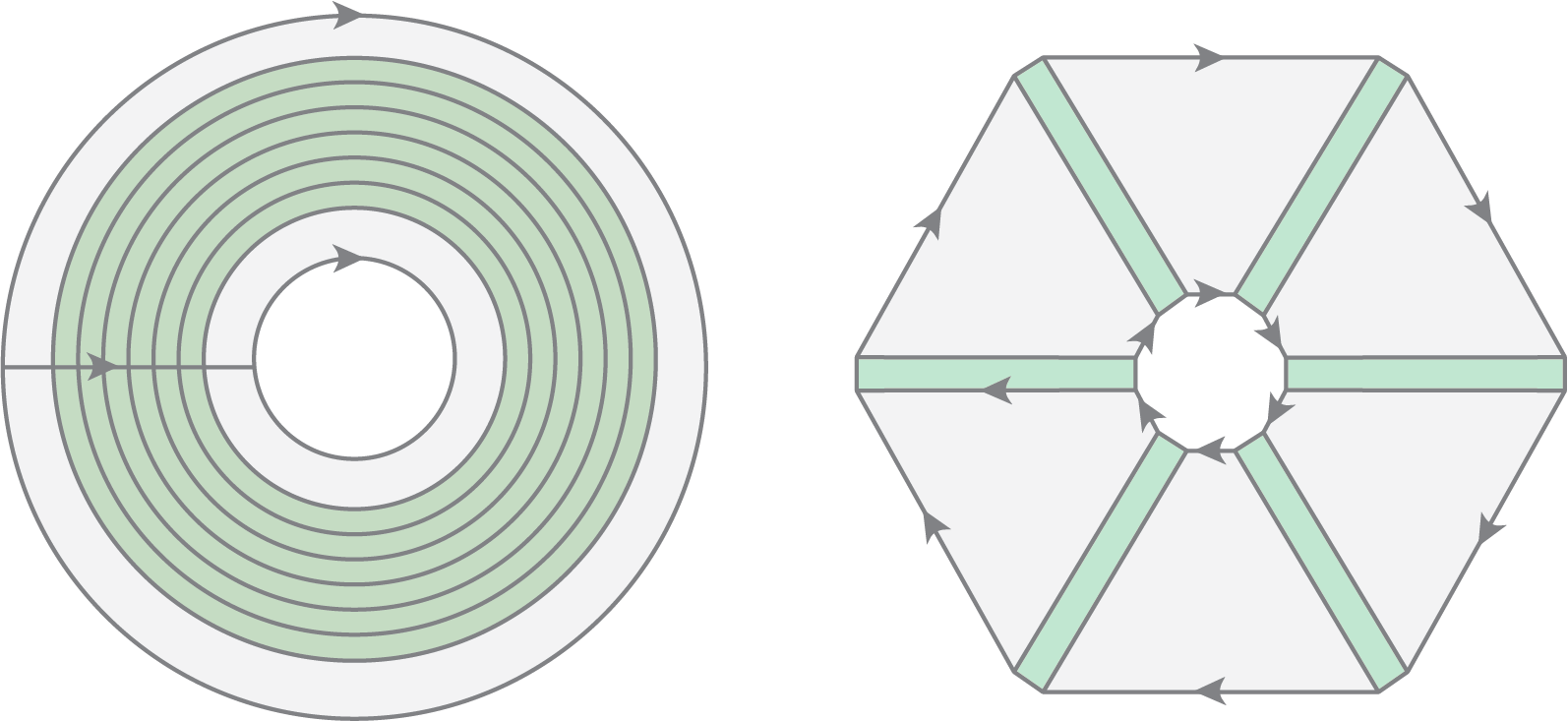}
 \put(150,160){\small{$\Delta$}}    
 \put(102,110){\small{$\Delta_0$}}    
 \put(85,176){\small{$u$}}    
 \put(85,105){\small{$v$}}    
 \put(21,75){\small{$W$}}    
 \put(360,160){\small{$\Delta$}}    
 \put(240,68){\small{$c^M$}}    
 \put(215,125){\small{$u_1$}}    
  \put(290,171){\small{$u_2$}}      
 \put(367,125){\small{$u_3$}}     
 \put(367,40){\small{$u_4$}}     
 \put(293,-4){\small{$u_5$}}      
 \put(215,40){\small{$u_6$}}      
 \put(284, 90){\small{$v_1$}}    
  \put(292,96){\small{$v_2$}}      
 \put(302,90){\small{$v_3$}}     
 \put(301,76){\small{$v_4$}}     
 \put(292,72){\small{$v_5$}}      
 \put(283,77){\small{$v_6$}}  
 \put(197,82){\small{$t_{i_1}^{\varepsilon_1}$}}
 \put(243,167){\small{$t_{i_2}^{\varepsilon_2}$}}
 \put(340,167){\small{$t_{i_2}^{\varepsilon_3}$}}
 \put(387,82){\small{$t_{i_2}^{\varepsilon_4}$}}
 \put(342,0){\small{$t_{i_2}^{\varepsilon_5}$}}
 \put(242,0){\small{$t_{i_2}^{\varepsilon_6}$}}

\end{overpic}
 \caption{Illustrating our proof of Theorem~\ref{t:cl2}}
  \label{fig:annuli}
\end{figure}

 If there are $t_i$-corridors in one of complementary subdiagrams,  $\Delta_0$ say,
 then this argument needs only a slight adjustment.  
 We focus now on the subdiagram of $\Delta_0$ that contains the inner boundary cycle $\rho$ of the
 stack of $t_i$-rings and is a connected component of the diagram obtained by deleting the interiors of
 all $t_i$-corridors from $\Delta_0$.  
 The path in the boundary of this subdiagram that  cobounds  with $\rho$
 is a concatenation of the sides of $t_i$-corridors (which are labelled by powers of the words $c_i,$ and $c_i'$ in the
 generators of $A$) interspersed by subwords of $v$ that contain only generators of $A$.
  When read  from a vertex at the
 end of a $t_i$-corridor,  the label $v^*$  on this path is equal as an element of $A$ to  $v$,
because these words cobound a disc diagram.  We treat the subdiagram adjacent
 to the outer boundary cycle of $\Delta$ similarly, and we are reduced to the previous case (that of no $t_i$-corridors)
 once we have passed from the elements  $u\sim v$ to  the elements $u^*\sim v^*$ via cyclic permutations.
 This time we conclude   that $\CL(u,v) \le 2m + (|u|+|v|)$. 
   
It remains to consider the case of a minimal annular diagram $\Delta$ that contains no $t_i$-annuli.  If there are no
radial $t_i$-corridors, then  $u,v\in A$ and $\Delta$ shows that they are conjugate, hence equal, in $A$.
 Assume, then,  that there are radial corridors in $\Delta$.  
These radial corridors correspond to the stable
letters in the HNN normal form of the elements $u$ and $v$.  
We now write $u$ and $v$ as formal products (HNN normal form) with occurrences of the stable letters separated
by elements of $A$ (so arcs between corridors  in the boundary of the diagram
 are now labelled by elements  of $A$ not by words).  Note that although $d_{G}(1,u)\le n$, 
the obvious bound we have on the lengths of geodesic words in $A$ representing each these elements
is $k^\alpha n^\alpha$, because of the distortion of $A$ in $G$. 

Drawing the annular diagram as illustrated right in Figure~\ref{fig:annuli}, 
we consider the slab regions between the radial corridors. A key point to 
observe is that if the sides of a pair of bounding corridors are labelled by words from cyclic groups
that intersect trivially,   then the
corresponding lines in $\E^r$ are skew, 
so since the words on the inner and outer cycles at the top and bottom of
the slab have length at most $k^\alpha n^\alpha$ (in the generators of $A$), the sides of the corridors can have length at most
$K n^\alpha$ in $A$, where $K>0$ is a constant  whose existence follows from Lemma~\ref{l:skew}. 
Taking further account of the assumed distortion of $A$,
this tells us that the distance in $G$ between the ends of the corridor (hence the inner and outer cycles of the annular diagram)
is bounded by a constant times $n$. This is the bound we seek in this case.

The only remaining case is where the sides of each slab are parallel,  meaning that the sides of the
two corridors bounding the slab are labelled by  elements from the same cyclic subgroup of $A$. 
We shall argue that in this case,   $u$ and $v$ have cyclic permutations that are equal as elements of $G$,
and hence they are conjugate by an element of length less than $(|u|+|v|)/2$.
The cyclic permutations in question are those  read from the ends of one side of a radial $t_{i}$-corridor;
if $u^* = t_{i_1}^{\varepsilon_{i_1}} u_1  \cdots  t_{i_r}^{\varepsilon_{i_r}} u_r$ and $v^* = t_{i_1}^{\varepsilon_{i_1}} v_1 \cdots t_{i_r}^{\varepsilon_{i_r}} v_r$ are these  cyclic permutations, where the radial $t_i$-corridors in $\Delta$ connect the $t_{i_j}^{\varepsilon_{i_j}}$ in $u^*$ to that in $v^*$, then $c^{M}u^* c^{-M}=v^*$  where $c$ is the
generator of an associated  subgroup of $t_{i_1}$. (Explicitly, if $\varepsilon_{i_1} =1$ then $c=c_{i_1}$, and if $\varepsilon_{i_1} =-1$  then $c=c'_{i_1}$.)    
The hypothesis that the sides of each slab are parallel is equivalent to the algebraic condition that if
 the stable letters of the (radial) corridors appear in $u$ and $v$ appear in the cyclic order $t_{i_1}^{\varepsilon_{i_1}}, \dots, t_{i_r}^{\varepsilon_{i_r}}$,  then the consecutive associated subgroups coincide in the following sense:  
 with indices mod $r$, if   $\varepsilon_{i_j} = \varepsilon_{i_{j+1}} = 1$ then  $L'_{i_j} = L_{i_{j+1}}$ and so $c'_{i_j} = c^{\pm 1}_{i_{j+1}}$;   the other three possibilities for $\varepsilon_{i_j}, \varepsilon_{i_{j+1}} \in \{\pm 1 \}$ are similar.

It follows  that $c$ actually commutes with $u^*$,  and therefore $u^*=v^*$,  as claimed.  In more detail, $c^{\pm M}u^* c^{\mp M} = u^*$ because $c_{i_j}^{\pm M} t_{i_j} {c'}_{i_j}^{\mp M} = t_{i_j}$, and  ${c'}_{i_j}^{\pm M} t_{i_j}^{-1} c_{i_j}^{\mp M} = t_{i_j}^{-1}$, and  $c^{\pm M}_{i_j} u_j c_{i_j}^{\mp M} = {c'}_{i_j}^{\pm M} u_j {c'}_{i_j}^{\mp M} = u_j$, and we
have the consecutive equalities described above between the $c_{i_j}$ or $c'_{i_j}$ and the $c_{i_{j+1}}$ or  $c'_{i_{j+1}}$.
\end{proof}

\section{The distortion of the centre in $\tbpqp$} \label{s:dist of centre in tbpqp}

Throughout, we work with the presentations of $\bpqp$ and  $\tbpqp$ given in the introduction.

\begin{lemma}\label{l:onlyTh}
Let $w$ be a word in the generators of $\bpqp$.  Suppose that $w=1$ in $\bpqp$ and that there is a 
van Kampen diagram with boundary label $w$ that contains exactly $M$ 2-cells labelled $[b,\th]^{\pm 1}$. Then
$w=z^m$ in $\tbpqp$ with $|m|\le M$.
\end{lemma}

\begin{proof}
The standard proof of van Kampen's Lemma translates a van Kampen diagram $\Delta$ with boundary label $w$
into an equality in the free group on the generators,
\begin{equation}\label{e:free}
w = \prod_{i=1}^A u_i^{-1}\rho_i u_i,
\end{equation} 
with each $\rho_i$ a defining relation or its inverse; the integer
 $A$ is the number of 2-cells in $\Delta$ and the list $\rho_1,\dots, \rho_A$ records the boundary labels 
 on these 2-cells. In $\tbpqp$,  each $\rho_i$ that is not $[b,\th]$ or its inverse
equals the identity element, whereas   $[b,\th] = z$.  Thus, in $\tbpqp$,  each factor of the product in the
equality (\ref{e:free}) reduces to the identity or else to $z^{\pm 1}$ (with $u_i^{-1}zu_i=z$ in the latter case, because
$z$ is central).  Exactly $M$ factors reduce to $z^{\pm 1}$,  so when we  cancel copies of
$z$ with $z^{-1}$, we are left with $w=z^m$ where $|m|\le M$.
\end{proof}

\begin{proposition}\label{p:new}
For all $w\in F(a,b,s,t,\th )$,
$$ w =1 \text{ in } \bpqp \ \iff \ w = z^N \text{ in } \tbpqp \text{ with } |N| \le \k^\alpha |w|^{\alpha +1},$$ 
where  $\kappa >1$ is the constant of Proposition~\ref{p:distort}. 
\end{proposition}

\begin{proof}
The implication ($\Leftarrow$) is trivial. For the implication ($\Rightarrow$),  suppose $|w|=n$ and
$w=1$ in $\bpqp$, hence $w=z^N$ in $\tbpqp$; we must bound $|N|$. Consider a minimal-area  
van Kampen diagram $\Delta$ for $w$ over our fixed presentation  (\ref{e:bpqp}) for $\bpqp$.  The only relation involving $\th$
is $[\th, b]=1$  and all of the 2-cells with this label are contained in $\th$-corridors.  Each such
corridor begins and ends on the boundary of the diagram, so the number of these corridors 
is at most $|w|/2$.  Each side of a $\th$-corridor is a path labelled by a word $b^R$
where $R$ is the area of the corridor,   and in $\bpqp$
this word defines the same element as either of the two arcs in the boundary cycle of $\Delta$ that have the
same endpoints as the path.
Each of these arcs has length less than $n$, so $d(1,b^R)< n$ in  $\bpqp$. 
Moreover,  as $b$ lies in the retract $\bpq$ obtained by killing $\th$, we have $d(1,b^R)< n$ in  $\bpq$ (with our fixed
choice of generators).  It therefore follows from Proposition \ref{p:distort} that $|R|\le \k^\alpha n^\alpha$. 
Thus the area of each $\th$-corridor in $\Delta$ is at most $\k^\alpha n^\alpha$.
There are fewer than $n$ corridors, so the number of 2-cells in the diagram labelled $[b,\th]^{\pm 1}$ is 
 less than $\k^\alpha n^{\alpha+1}$. 
Lemma~\ref{l:onlyTh} tells us that $|N|$ is bounded above by this number. 
\end{proof}

\begin{proposition} \label{pp:dist-z}
The distortion of the centre $Z=\<z\><\tbpqp$ is  $\dist_{Z}(n) \simeq n^{\alpha+1}$.
\end{proposition}

\begin{proof}
Proposition \ref{p:new} shows that $\dist_{Z}(n) \preceq n^{\alpha+1}$.  
To obtain a complementary lower bound, we 
choose a geodesic word $\omega$ representing $b^{\lfloor n^\alpha\rfloor}$ in $\bpq$.  Proposition~\ref{p:distort} tells us that
$|\omega| \le \k n$. Thus  $W_n:=\omega^{-1}\th^{-n}\omega\th^n$ has length at most $2n(\k+1)$.
The equality  $\omega  = b^{\lfloor n^\alpha\rfloor}$ in $\bpq$ remains valid in  $\tbpqp$ because the defining relations for $\bpq$ are all among the defining relations for $\tbpqp$.  
So in $\tbpqp$ we have
$$
W_n = [ b^{\lfloor n^\alpha\rfloor},\, \th^n] = z^{n \lfloor n^\alpha\rfloor}.
$$
As $|W_n|\le 2n(\k+1)$, this establishes the desired lower bound $\dist_{Z}(n) \succeq n^{\alpha+1}$.
\end{proof}

\begin{remark}\label{r:all}
The integer $n \lfloor n^\alpha\rfloor$ that appeared at the end of the preceding proof,  which we
now call $M_0$,  lies between 
$n^{\alpha +1} -n$ and $\lceil n^{\alpha +1}\rceil$.  If $M$ is any other integer in this range, then $M = M_0+\e$
with $|\e|\le n$, so $Wz^\e$ is a word of length less than $2n(\k +2)$ that equals $z^M$ in $\tbpqp$.
\end{remark} 

\begin{lemma} \label{l: lower bound on CL}
$\CL_{\tbpqp} (n) \succeq n^{\alpha +1}$.
\end{lemma}

\begin{proof}
From Remark~\ref{r:all} we know that if $M=\lceil n^{\alpha +1}\rceil$ then $d(1, z^{M}) < 2n(\k+2)$
in $\tbpqp$, so $d(1, bz^{M})\le 2n(\k+2)$.  The unique shortest word conjugating $b$ to $bz^{M}$ in the Heisenberg group $H = \langle b, \theta, z \mid z \text{ central}, \  [b,\theta]= z \rangle$ is $\th^M$.  Since killing $a$, $b$ and $s$ retracts $\tbpqp$ retracts onto $H$, the word $\th^M$ is also the unique shortest conjugator in $\tbpqp$.     
And $d(1, \th^M) = M$  in $\tbpqp$ because killing the other generators retracts $\tbpqp$ onto $\<\th\>$. 
\end{proof}

The proof that  
$\CL_{\tbpqp} (n) \preceq n^{\alpha +1}$ is much more involved and will occupy  the remainder of this article.

\section{Conjugacies of the form $g\sim g z^N$} \label{s: special conjugacies}

Our strategy for showing that $\CL_{\tbpqp} (n) \preceq n^{\alpha +1}$ will be to reduce to an analysis of conjugacies $\gamma\sim\gamma z^N$
in $\tbpqp$.  In this section and the next, we examine the intimate connection between these conjugacies and the structure of 
centralisers in $\bpqp$.  The first part of our discussion applies to arbitrary central extensions
$$
1 \to Z \to \wt{G} \to G\to 1,
$$
with $Z\cong\Z$.  We fix a generator $z$ for $Z$.

\begin{definition} \label{d:zeta}[The maps $\zeta_g$]  Let $g\in G$.
For each element of the centralizer $x\in C_{G}(g)$ 
and all preimages $\tilde{x}, \tilde{g} \in\wt{G}$
we have $\tilde{x}^{-1}\tilde{g}\tilde{x} = \tilde{g} z^{m}$ in $\wt{G}$, where $m\in\Z$ is independent of the
choices of $\tilde{x}$ and $\tilde{g}$ because different choices differ by a power of $z$, which is central.
Define 
$$
\zeta_g : C_G(g) \to \Z
$$
by  $\zeta_g(x) :=m$.
\end{definition}

The following basic facts are easily verified.

\begin{lemma}\label{l:zeta-basic} For all $g,  x\in G$ and $h\in C_G(g)$, 
\begin{enumerate}
\item \label{l:zeta-basic1} $\zeta_g$ is a homomorphism;
\item \label{l:zeta-basic2} $\zeta_g(g)=0$;
\item \label{l:zeta-basic3} if $C_G(g)$ is cyclic then $\zeta_g$ is the zero map;
\item \label{l:zeta-basic4} $\zeta_g(h)=-\zeta_h(g)$;
\item \label{l:zeta-basic5}  $\zeta_{x^{-1}gx}(x^{-1}hx) = \zeta_g(h)$.
\end{enumerate}
\end{lemma}

 The following lemma is useful when the cyclic subgroup generated by $x_0$
is heavily distorted, because in such a situation,  one can   
 concentrate the word length of conjugators $W=\wt{x}_0^{\lambda}\wt{x}_1^{\mu_1}\dots \wt{x}_r^{\mu_r}$ on the
 syllable $\wt{x}_0^\lambda$ when trying to minimise $d_{\wt{G}}(1,W)$.

\begin{lemma}\label{l:bezout}
If $\wt{g}\in\wt{G}$ has image $g\in G$ and $C_G(g)$ is generated by $\{x_0,\dots,x_r\}$ with $\zeta_g(x_i)=m_i$,
and $m_0\neq 0$, then for all $N\in\Z$ and all preimages $\wt{x}_i \in \wt G$,  if $\wt{g}\sim \wt{g}z^N$ then
\begin{equation}
(\wt{x}_0^{\lambda}\wt{x}_1^{\mu_1}\cdots\wt{x}_r^{\mu_r})^{-1} \  \wt{g}\ (\wt{x}_0^{\lambda}\wt{x}_1^{\mu_1}\cdots
\wt{x}_r^{\mu_r})  =  \wt{g} z^N
\end{equation}
for some $\lambda, \mu_1, \ldots, \mu_r \in \Z$ with  $|\lambda| <  |N/m_0|  + \sum_{i=1}^r |m_i|  \text{ and } |\mu_i| < |m_0|$ for $i=1,\dots,r$.
\end{lemma}

\begin{proof} 
If $\wt{g}\sim \wt{g}z^N$ then $N\in \im \zeta_g$, which is generated (as an additive subgroup of $\Z$) by 
$m_0,\dots,m_r$, so there exist integers $a_i$ such that
$$
a_0 m_0 + \cdots + a_r m_r = N.
$$
For $i=1,\dots,r$ we write $a_i = \eta_i m_0 + \mu_i$ (in integers) with $0\le \mu_i < |m_0|$.  Then,
$$
\left(a_0 + \sum^r_{i=1}\eta_i  m_i \right) \, m_0 + \sum_{i=1}^r \mu_i m_i  = N,
$$
and
$$
\left|a_0 + \sum_{i=1}^r\eta_i  m_i \right| \le \left| N/m_0 \right| + \sum_{i=1}^r \left| \mu_im_i/m_0 \right| < |N/m_0|  + \sum_{i=1}^r |m_i |.
$$
To complete the proof, we define $\lambda = a_0 + \sum^r_{i=1}\eta_i  m_i$ and note that 
conjugation by $\wt{x}_i^j$ sends $\wt{g}$ to $\wt{g}z^{jm_i}$.
\end{proof}

\section{Centralisers and zeta-maps for $\bpqp$} \label{s: centralizer and zeta-maps}
 
\def\t{\tau}

To clarify the salient  points in our discussion of centralisers in $\bpqp$,  we begin with a more general setting.

\begin{lemma}\label{l:HNN-cent} Let $G = H\dot{\ast}_A$ be the HNN extension in which the
stable letter $\t$ commutes with the associated subgroup $A<H$. Assume that $A$ is  abelian
and that $h^{-1}Ah\cap A= \{ 1 \}$ for all  $h\in H\ssm C_H(A)$.   
\begin{enumerate}
\item \label{l:HNN-cent1} If $g\in H$ is not conjugate to an element of  $A$ then $C_G(g)=C_H(g)$.
\item \label{l:HNN-cent2} If $g\in A\ssm \{1\}$ then  $C_G(g)=  \< C_H(A), \, \t\> = C_H(A)\dot{\ast}_A$.
\item \label{l:HNN-cent3} If $g\in  \< C_H(A), \, \t\>$ is not conjugate to an element of $C_H(A)$,  then $C_G(g)  \cong A\times \Z$,
where  $1\times\Z$ is generated by a maximal root of $ga$ for some $a\in A$.
\item \label{l:HNN-cent4} If $g$ is not conjugate to an element of $H\cup \< C_H(A), \, \t\>$,  then $C_G(g)$ is cyclic.
\item \label{l:HNN-cent5} In $G$,  if $x^{-1}Ax\cap A\neq \{ 1 \}$,  then $x\in C_G(A)=  \< C_H(A), \, \t\>$.
\end{enumerate}
\end{lemma}

\begin{proof}
Consider the action of $G$ on the Bass-Serre tree $T$ of the splitting $G = H\dot{\ast}_A$. Items \eqref{l:HNN-cent1} and \eqref{l:HNN-cent2}
cover the elliptic elements of this splitting (up to conjugacy).  In case \eqref{l:HNN-cent1},  the fixed point set of $g$ is the single
vertex $H\in G/H={\text{Vert}}(T)$, so the centraliser of $g$ must lie in the stabiliser of this vertex, which is $H$.  

Case \eqref{l:HNN-cent2}, covers edge stabilisers, which are the conjugates of $A$.  In this case we argue algebraically.
It is clear that if $x\in \< C_H(A), \, t\>$ then $[g,x]=1$, so we will be done if we can prove the
{\em claim} that $x^{-1}gx\in A$
only if $x\in \< C_H(A), \, \t\>$.  This claim will also establish item \eqref{l:HNN-cent5}. 

We prove the claim by induction on the HNN length of $x$,  
which is the number of stable letters $\t^{\pm 1}$ present when
we  write $x$ in reduced HNN form
$$x = h_0 \t^{e_1} h_1\t^{e_2}\cdots h_{m-1} \t^{e_m} h_m,$$
with $h_i\in H\ssm A$ for $0<i<m$ and all $e_i\neq 0$.   The induction begins with the case $m=0$, which 
is covered by the hypothesis that $h^{-1}Ah\cap A =1$ if $h\in H\ssm C_H(A)$.  
For the inductive step,  we consider the cancellation that brings $x^{-1}gx$ into reduced form.
If  the middle term $h_0^{-1} g h_0$ in the naive product does not represent an element
of $A$, then $x^{-1}gx$ is already in reduced form and it is not an element of $A$.
 If $h_0^{-1} g h_0\in A$, then $h_0^{-1}Ah_0\cap A\neq 1$,
so $h_0\in C_H(A)$, by hypothesis.  In this case,   $x^{-1}gx = x_0^{-1}gx_0$
where $x_0:=h_1 \t^{e_2}\dots h_{m-1} \t^{e_m} h_m$  is covered by the inductive hypothesis. This completes the
proof of \eqref{l:HNN-cent2} and \eqref{l:HNN-cent5}.

Items \eqref{l:HNN-cent3} and \eqref{l:HNN-cent4} cover hyperbolic isometries of $T$.  If $g\in G$ acts as a 
non-trivial hyperbolic isometry, then it has 
a unique axis of translation and its centraliser $C_G(g)$ must preserve this axis, acting
 by translations on it. Case \eqref{l:HNN-cent4} is the case where the action of $C_G(g)$ is faithful, which forces  $C_G(g)$ to be
 cyclic. Case \eqref{l:HNN-cent3} is where the action of  $C_G(g)$ has a non-trivial
 kernel: the kernel fixes every edge in the axis so, 
 after conjugating, we may assume that it is contained in $A$. Then \eqref{l:HNN-cent5} implies that $g$ centralises $A$,
 so the kernel is the whole of $A$.  A second application of (v) tells us that since $A$ is normal in $C_G(g)$, it
 must be central.  Thus we have a central extension $1\to A\to C_G(g)\to \Z\to 1$, whence
 $C_G(g) \cong A\times \Z$ where the second factor is generated by an element that acts as a translation of
 minimal displacement on the axis, which will be an $m$-th root of $ga$ for some $a\in A$ with $m$ maximal.
\end{proof}

The following observation concerning our presentations (\ref{e:bpqp}) for $\bpqp$ and (\ref{e:tbpqp}) for $\tbpqp$ will be important for the proposition that follows.

\begin{lemma}\label{l:lift}
For all words $u,v\in F(a,b,s,t)$, if $u=v$ in $\bpqp$ then $u=v$ in $\tbpqp$.
\end{lemma}

Lemma \ref{l:zeta-basic}\eqref{l:zeta-basic5} assures us that the following description of the zeta maps for $\bpqp$ is exhaustive.

\begin{proposition}\label{p:zetas} In $\bpqp$, we have 
$C_{\bpqp}(b) = \<b\> \times F$, where $F=\<a,\th\>$ is free of rank $2$.  Furthermore, 
\begin{enumerate}
\item \label{p:zetas2} if no conjugate of $\gamma$ lies in $C_{\bpqp}(b)$, then
$\zeta_\gamma : C_{\bpqp}(\gamma)\to \Z$ is the zero map;
\item \label{p:zetas3} if $\gamma = b^l$ for some non-zero $l \in \Z$,   then $C_{\bpqp}(\gamma) =\<b\> \times F \to \Z$  and the image of $\zeta_\g$ is generated by $\zeta_\g(\th ) = l$;
\item \label{p:zetas4} If $\gamma = b^l \omega$ with     $\omega\in F\ssm\{1\}$, then  $C_{\bpqp}(\gamma)
=  \<b\> \times \< \omega_0 \>$, where $\omega_0$ is a maximal root of $\omega$ in $F= \< a , \theta \>$, and 
the image of $\zeta_\g$  is
generated by 
$\zeta_\gamma(b) = -j$ and  $\zeta_\g(\omega_0 )= lj_0$, where  $j$ (resp.~$j_0$)
is the exponent sum of $\th$ in $\omega$ (resp.~$\omega_0$).
\end{enumerate}   
\end{proposition}

\begin{proof}
We will apply the analysis of Lemma~\ref{l:HNN-cent} with $G=\bpqp,\, H=\bpq,\ A=\<b\>$ and $\t=\th$.  This is valid
because $x \in B_{pq}$ satisfies  $x^{-1}\<b\>x\cap \<b\> \neq \{ 1 \}$ if and only if $x\in \<a,b\> = C_{\bpq}(b)$.  To see this,  observe that
$b$ belongs to the vertex group $\Z^2=\<a,b\>$ of the 2-edge splitting that defines $\bpq$ and that
$\<b\>$ intersects the edge groups $\<a^q\>, \, \<a^pb\>, \, \<a^pb^{-1}\>$ trivially; it follows that the fixed point set of 
$b$ in the Bass-Serre tree of the splitting is the unique vertex with stabiliser $\<a,b\>$, and $C_{\bpq}(b)$
has to preserve this fixed point.  That  $C_{\bpqp}(b) = \<b\> \times F$ is now a special case of  Lemma~\ref{l:HNN-cent}\eqref{l:HNN-cent2}. 

For \eqref{p:zetas2},  note that if
 $\gamma$ is not conjugate to any element of $C_{\bpqp}(b)$,   then either Lemma~\ref{l:HNN-cent}\eqref{l:HNN-cent4} applies
and  $C_{\bpqp}(\g)$ is cyclic, or else  $\g$ is conjugate to 
an element $\gamma'\in \bpq\ssm \<b\>$. In the former case, Lemma~\ref{l:zeta-basic}\eqref{l:zeta-basic3} applies and tells us that $\zeta_{\gamma}$ is the zero map.   
In the latter case,    $C_{\bpqp}(\g') < \bpq$, 
by Lemma~\ref{l:HNN-cent}\eqref{l:HNN-cent1},   and Lemma~\ref{l:lift} tells us that in this case $\zeta_{\gamma'}$ is the zero map,
and hence so is $\zeta_{\gamma}$,   by Lemma~\ref{l:zeta-basic}\eqref{l:zeta-basic5}.

For \eqref{p:zetas3},  the case  $\gamma = b^l$,  we have
 $C_{\bpqp}(\g) = \<a, b, \th\>$ and $\zeta_\gamma(a)=\zeta_\gamma(b)=0$
while $\zeta_\gamma(\th) = l$.  

For \eqref{p:zetas4},  the case $\gamma=b^l \omega$,  Lemma~\ref{l:HNN-cent}\eqref{l:HNN-cent3}  tells us that
$C_{\bpqp}(\g) = \<b\> \times \<\omega_0\>$  where $\omega_0$ is a maximal root of $\omega$ in $F=\<a,\th\>$. 
As $\omega_0$ remains a root of $\omega$ in $\tbpqp$,  it commutes with any preimage of $\omega$ and
hence $\zeta_\gamma(\omega_0) = \zeta_{b^l}(\omega_0) = lj_0$, where $j_0$ is the exponent sum of $\th$ in 
$\omega_0$.  (Here we have used the fact that $ \zeta_{b^l}(a)=0$ and  $\zeta_{b^l}(\th)=l$.)
Similarly,   $\zeta_\gamma(b) = \zeta_{\omega}(b) =- j$.
\end{proof}

\section{Preferred representatives of conjugacy classes in $\bpqp$} \label{sec: preferred reps}

 \begin{proposition} \label{p:preferred new}  
If the conjugacy class $[\gamma]$
of $\gamma \in\bpqp$ intersects $C_{\bpqp}(b)$, then there exists a word $u_0 \in F(a, b, s,  t, \th)$ that represents an element of $[\gamma]\cap C_{\bpqp}(b)$ and  satisfies   
 \begin{equation}\label{e:length of u_0 new}
|u_0| \le \ d_{\bpqp} (1, \g).
\end{equation}
 \end{proposition}

\begin{proof}
	Proposition~\ref{p:zetas} tells us that $C_{\bpqp}(b)$ is $\<b\> \times F(a,\th)$.  So, given our hypotheses, $\gamma \sim v$ in $\bpqp$ for some word $v$ on $a, b, \theta$.   
		Consider a word $u\in F(a,b,s,t,\th)$  of minimal length among all words representing  elements of $[\gamma]$.  It need not be the case that $u$ represents an element of $C_{\bpqp}(b)$,  but we will argue that some cyclic conjugate $u_0$ of $u$ does. Then we will have $|u_0| = |u| \le  \ d_{\bpqp} (1, \g)$, as required. 

 \begin{figure}[ht]
\begin{overpic}
{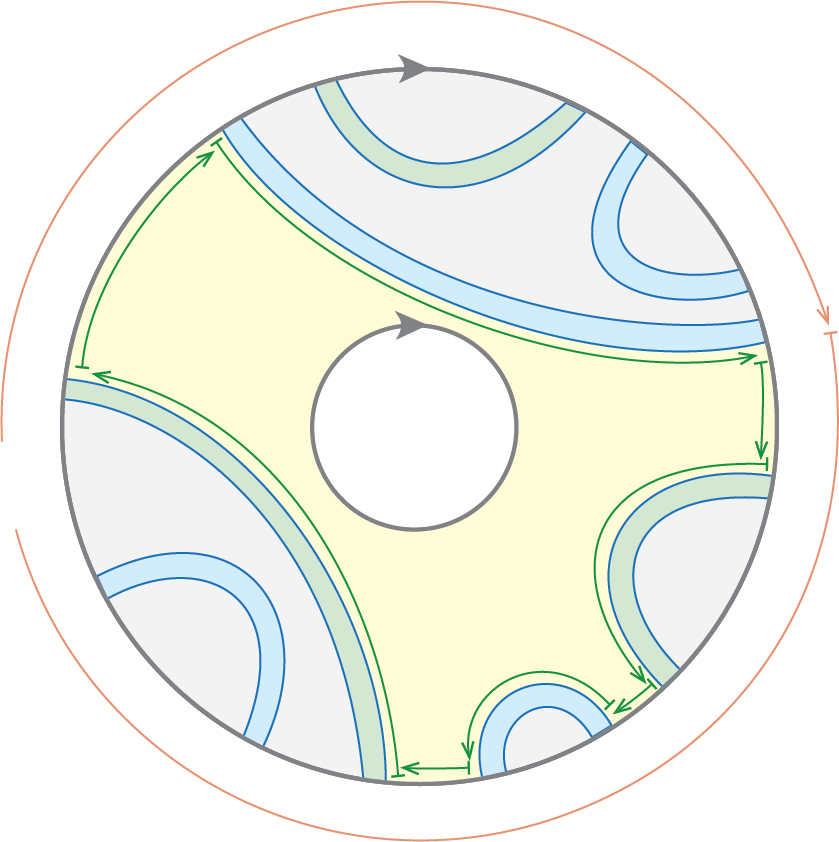}
 \put(200,180){\small{$\Delta$}}    
 \put(96,192){\small{$u$}}    
 \put(-1,84){\small{$u_0$}}    
 \put(187,118){\small{$\ast$}}    
 \put(174,102){\small{$u_1$}}    
  \put(142,39){\small{$u_2$}}      
 \put(100,23){\small{$u_3$}}     
 \put(32,136){\small{$u_4$}}     
  \put(137, 80){\small{$v$}}    
 \put(137, 80){\small{$v_1$}}    
  \put(118,45){\small{$v_2$}}      
 \put(86.5,60){\small{$v_3$}}     
 \put(72,134){\small{$v_4$}}     
 \put(98,103){\small{$v$}}
\end{overpic}
 \caption{Illustrating the   proof of Proposition~\ref{p:preferred new}}
  \label{fig:annuli}
\end{figure}
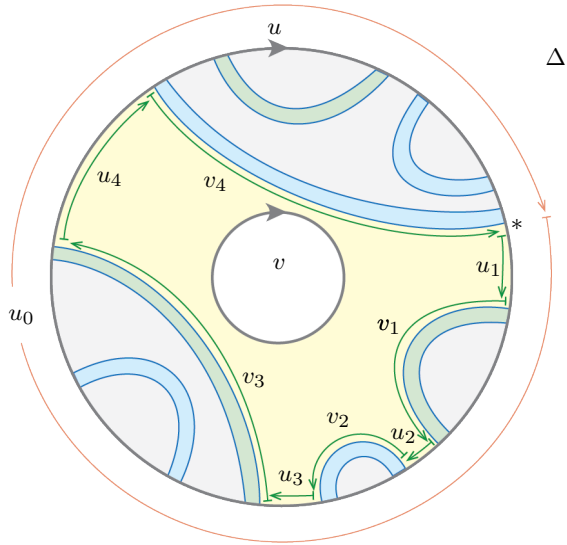

Let $\Delta$ be an annular diagram  portraying the conjugacy from $u$ to $v$, as illustrated in Figure~\ref{fig:annuli}.  Any $s$- or $t$-corridor in $\Delta$ must have both of its ends on the boundary component labelled $u$ because there are no $s$- or $t$-letters in $v$.  (There may also be $s$- and $t$-annuli in $\Delta$.)  Let $\Delta_0$ be the minimal annular subdiagram  of $\Delta$ that has the same outer cycle as $\Delta$ and contains all of the $s$- and $t$-corridors.  The word around the inner boundary cycle, read from the terminal vertex $\ast$ of the end of one of the $s$- or $t$-corridors, will be a concatenation
$$
U= u_1v_1\cdots u_mv_m
$$
in which each $u_i$ is a subword of $u$  (so contains only letters $a$, $b$ and $\th$)
  and each $v_i$  labels  the side of an $s$- or $t$-corridor (and so contains only letters $a$ and $b$).  Thus $U$ is a word in $a$, $b$ and $\th$,  and so represents an element of $C_{\bpqp}(b)$. And $\Delta_0$ demonstrates that the cyclic conjugate $u_0$ of $u$ read from $\ast$ equals $U$ in $\bpqp$, as required.  
\end{proof}

\section{Estimating the word metric in $\bpqp$} \label{s:word metric in bpqp} 

 Each element of $C_{\bpqp}(b) =  \<b\> \times \<a,\th\> $ is represented by a unique word of the form
\begin{equation}\label{e:deco}
\sigma \equiv b^l  a^{n_0}\th^{e_1}a^{n_1} \cdots \th^{e_r}a^{n_r}
\end{equation}   
where for all $i$ we have $e_i \in \{-1, 1\}$ and $l, n_i \in\Z$ with $n_i\neq 0$ if $e_i=-e_{i+1}$.
Proposition~\ref{p:distort} gave us an understanding of lengths of elements of $\langle a, b \rangle$ in $\bpq$.   Given that the distances between elements of $\langle a, b \rangle$ in $\bpq$ and in $\bpqp$ are equal (because killing $\theta$ retracts $\bpqp$ onto $\bpq$), the following proposition, together with a reverse bound provided by the triangle inequality,  promotes this  to an understanding of the lengths of elements of $ C_{\bpqp}(b)=\<a,b,\th\>$ in $\bpqp$.

\begin{proposition} \label{p: length of H in bpqp}
There exists $C \geq 1$ such that for all   $\sigma$   per \eqref{e:deco}, 
\begin{equation}\label{e:deco2}
 d_{\bpq}(1,b^l) + r + \sum_{i=0}^r   d_{\bpq}(1,a^{n_i})     \leq  C d_{\bpqp}(1, \sigma). 
 \end{equation}   
\end{proposition}

 \begin{figure}[ht]
\begin{overpic}
{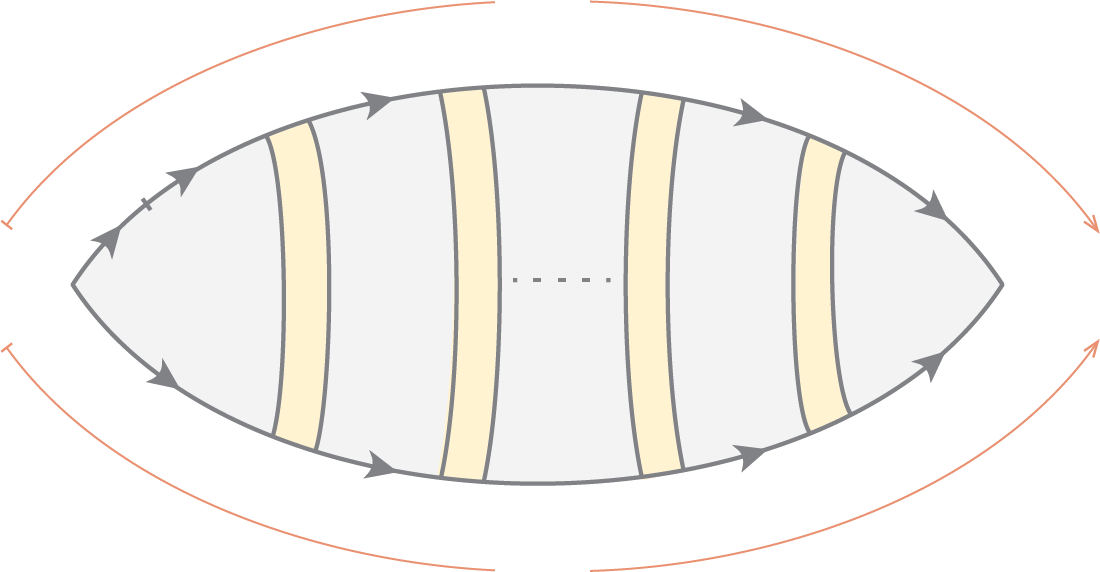}
 \put(127,-2){\small{$w$}}        
 \put(27,40){\small{$w_0$}}    
 \put(84,16){\small{$w_1$}}    
 \put(179,18){\small{$w_{r-1}$}}     
 \put(230,43){\small{$w_r$}}     
 \put(63,110){\small{$\theta^{e_1}$}}    
 \put(106,121){\small{$\theta^{e_2}$}}    
 \put(157,119){\small{$\theta^{e_{r-1}}$}}     
 \put(197,107){\small{$\theta^{e_r}$}}     
 \put(127,135){\small{$\sigma$}} 
 \put(18,85){\small{$b^{l}$}}    
 \put(32,100){\small{$a^{n_0}$}}    
 \put(81,119){\small{$a^{n_1}$}}    
 \put(178.5,117){\small{$a^{n_{r-1}}$}}     
 \put(230,90){\small{$a^{n_r}$}}
  \put(63,22){\small{$\theta^{e_1}$}}    
 \put(106,11){\small{$\theta^{e_2}$}}    
 \put(156,13){\small{$\theta^{e_{r-1}}$}}     
 \put(199,27){\small{$\theta^{e_r}$}}         
\end{overpic}
 \caption{Illustrating our proof of Proposition~\ref{p: length of H in bpqp}}
  \label{fig:w_sigma}
\end{figure}
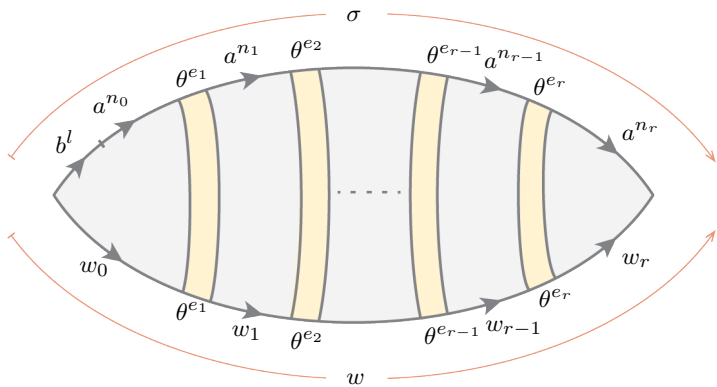

\begin{proof}
Choose a geodesic word $w$ that equals $\sigma$ in $\bpqp$ and consider a van Kampen diagram (per Figure~\ref{fig:w_sigma}) showing $w= \sigma$ in $\bpqp$.
Recall that $\bpqp$ is an HNN extension of $\bpq$ with stable letter $\th$.   
A $\th$-corridor emanates from each $\th$-edge on the boundary arc labelled $\sigma$ and
these corridors all end on the boundary arc labelled $w$. (A corridor  
cannot have both its ends   on the arc labelled $\sigma$
because then there would be a disc-subdiagram enclosed by an innermost such corridor and its presence would imply that 
a subword of  $\sigma$ defines an element of $\<b\>$, since the sides of $\th$-corridors are labelled
by powers of $b$, and the defining form of  $\sigma$ precludes this.) The labels on the subarcs between the ends of the $\th$-corridors are subwords
  $w_i$ of $w$ such that 
  $$
  w \equiv w_0\th^{e_1}w_1 \cdots \th^{e_r}w_r.
  $$
 A key point to observe is that, as group elements,  $w_i = a^{n_i} b^{l_i}$ for some $l_i \in \Z$,  because the 
 subdiagrams (which may be degenerate)
 between successive corridors are van Kampen diagrams with boundary labels of the form  $w_ib^*a^{n_i}b^*$, 
 where the powers of $b$ are the labels on the sides of the corridors. 
  From this decomposition we have
 \begin{equation}\label{e11}
 d_{\bpqp}(1,\sigma) = |w| = r + \sum_{i=0}^r |w_i| = r + \sum_{i=0}^r d_{\bpqp}(1, a^{n_i} b^{l_i}).
 \end{equation} 
  Now, $\sum_{i=0}^r l_i = l$ because killing   $\th$ retracts $\bpqp$ onto $\bpq$, and the image of $\sigma$
  is $b^la^{\sum n_i}$ while the (equal!) image of $w$ is $b^{\sim l_i}a^{\sum n_i}$ in $\<a,b\>\cong\Z^2$. Thus,
 \begin{equation} \label{e:bs}  
 d_{\bpq}(1,b^l) \leq \sum_{i=0}^r d_{\bpq}(1,b^{l_i}).
 \end{equation}
 Because the words metrics on    $d_{\bpq}$ and $d_{\bpqp}$ agree on $\langle a, b \rangle$, we can use 
 Proposition \ref{p:distort} to compare  the three terms of   
 \begin{equation}\label{e:help}
  d_{\A}(1, a^{n_i}) +  d_{\A}(1,  b^{l_i})  = d_{\A}(1,   a^{n_i} b^{l_i})
 \end{equation} 
with the corresponding distances in $\bpqp$.
To this end, we first use the concavity of 
$f(x)=x^{1/\alpha}$ to deduce that
\begin{equation}\label{e:concave}
  d_{\A}(1, a^{n_i})^{1/\alpha} +  d_{\A}(1,  b^{l_i})^{1/\alpha} \le  2^{1-\frac{1}{\alpha}} d_{\A}(1,   a^{n_i} b^{l_i})^{1/\alpha},
 \end{equation} 
then we use Proposition \ref{p:distort} to bound the terms on the left above and term on the right below,  concluding that  
 \begin{equation}\label{e:more-help} 
 d_{\bpqp}(1,  a^{n_i})   + d_{\bpqp}(1,  b^{l_i})  \le \k^{2}2^{1-\frac{1}{\alpha}}\,  d_{\bpqp}(1, a^{n_i} b^{l_i}) 
 \end{equation}  
 for $i = 0, \ldots, r$. 
  Summing and then calling on  \eqref{e11} and \eqref{e:bs} gives \eqref{e:deco2} for a suitable constant $C \geq 1$.
\end{proof} 

We will use  Proposition~\ref{p: length of H in bpqp} in the final stages of our proof of  Theorem~\ref{t:main} in Section~\ref{s:Proof of Main Theorem}.  More particularly, we will need:

\begin{corollary} \label{c: length in bpqp} There exists a constant $C>0$ such that
if $\sigma = b^l \omega$ in $\bpqp$, where $l \in \Z$  and  $\omega \in F(a,\th)$, then 
\begin{align}\label{e:big0}
 d_{\bpqp}(1,b^l) \leq C  d_{\bpqp}(1,\sigma).
\end{align}
And if, further, $\omega=\omega_0^m$ for some $m \geq 1$ and some reduced $\omega_0 \in F(a,\th)$, then 
\begin{align}\label{e:big2}
 d_{\bpqp}(1,\omega_0)  \leq C  d_{\bpqp}(1,\sigma). 
\end{align}
\end{corollary}

\begin{proof}
The inequality \eqref{e:big0} follows immediately from \eqref{e:deco2}.   
Towards \eqref{e:big2}, we assume that $\omega_0$ and $\omega$
are written as freely reduced words, which implies in particular
that $\omega$ is the suffix of $\sigma$ following $b^l$ in the decomposition \eqref{e:deco}.
Let $u$ be the maximal prefix of $\omega_0$ such that $\omega_0 \equiv u \pi  u^{-1}$,  with $\pi$ cyclically
reduced.  Then $\omega \equiv u \pi^m u^{-1}$.
Now $\pi \equiv   
   a^{n'_0}\th^{f_1}a^{n'_1} \cdots \th^{f_{r'}} a^{n'_{r'}}$
for some $r'$, some $f_i \in \{-1, 1\}$ and some $n'_i \in\Z$ with $n_i\neq 0$ if $f_i=-f_{i+1}$.
And because $u \pi^m u^{-1}$ is   reduced, if we delete the subword that starts with the first of the $m$ instances 
of $\th^{f_1}$ and ends immediately before the final instance of  $\th^{f_1}$,  then we get $\omega_0$.  Thus, by deleting  $b^l$ and  whole syllables $\th^{e_j}$ and $a^{n_j}$ from the right hand side of  \eqref{e:deco},  we get the reduced word $\omega_0$.  
Therefore,  by the triangle inequality,
\begin{equation}\label{e:roo}
d_{\bpqp}(1,  \omega_0) \le r'' +  \sum_{i\in S} d_{\bpqp}(1, a^{n_i})  
\end{equation} 
for some $r''< r$ and subset $S\subset\{0,\dots,r\}$.  Proposition~\ref{p: length of H in bpqp}
tells us that this last quantity is at most $C  d_{\bpqp}(1,\sigma)$.
\end{proof}

\section{Proof of Theorem \ref{t:main}} \label{s:Proof of Main Theorem}

Suppose words $u,v\in F(a, b, s, t, \th,z)$ of total length $|u| + |v| = n$ represent conjugate elements of $\tbpqp$.  Instead of  conjugating $u$ to $v$ directly in $\tbpqp$,  we delete the letters $z^{\pm 1}$ they contain  and consider their images  $\bar{u}$ and $\bar{v}$ 
in $\bpqp$. We fix a reduced
word $u_0 \in F(a, b, s,  t, \th)$ representing an element of the conjugacy class $[\bar{u}] = [\bar{v}]$ in $\bpqp$  that satisfies    
 \begin{equation} 
|u_0| \le \ \min \{ d_{\bpqp} (1, u), \ d_{\bpqp} (1, v) \}.
\end{equation}
If $[\bar{u}]$ intersects $C_{\bpqp}(b)$, then we call on  Proposition~\ref{p:preferred new}, which allows us to further assume that $u_0$  represents an element of $C_{\bpqp}(b)$.

Theorem~\ref{t:cl} tells is that   $\CL_{\bpqp}(n) \simeq n$. Accordingly, there is a constant $C_1$ and
words $x_u, x_v \in  F(a, b, s, t,  \th)$ such that  $x_u^{-1} \bar{u} x_u =  x_v^{-1} \bar{u} x_v = u_0$ in $\bpqp$ and
 \begin{equation} \label{e: xuxv bound}
\max \{ |x_u|, |x_v|\} \leq C_1 n.
 \end{equation}

Then $x_u^{-1}ux_u = u_0z^{N_u}$ and $x_v^{-1} v x_v = u_0z^{N_v}$ in $\tbpqp$ for some integers $N_u, N_v$.
Now, $x_u^{-1} u x_u u_0^{-1}$ and $x_v^{-1} v x_v u_0^{-1}$ are words of 
length at most $(2C_1 + 2)n$, so Proposition~\ref{p:new} tells us that 
\begin{equation} \label{e:NuNv bound}
\max \{ |N_u|,  |N_v|\} \le C_2 n^{\alpha   +1},
\end{equation} 
where $C_2 = \k^\alpha (2C_1 + 2)^{\alpha   + 1 }$.

Let $N= N_v - N_u$. Then  
\begin{equation} \label{e:N bound}
|N|   \le 2C_2 n^{\alpha + 1}.
\end{equation}
In $\tbpqp$ we have $u_0z^{N_u} \sim u_0z^{N_v}$,   so  $u_0  \sim  u_0z^N$ since $z$ is central.
Therefore $N$ is in the image of $\zeta_{u_0} : C_{\bpqp}(u_0) \to \Z$ (the zeta map defined in \ref{d:zeta}).
  And if $y^{-1} u_0 y =  u_0z^N$ in $\tbpqp$, then  $x_u y x_v^{-1}$ conjugates  $u$ to $v$ in $\tbpqp$.
 So,  from  \eqref{e: xuxv bound} we have that
\begin{equation*} \label{e:CL reduction}
\CL_{\tbpqp}(u,v) \le  \CL_{\tbpqp} (u_0, u_0z^N) + 2 C_1 n.
\end{equation*} 
Thus we will be done if we can establish  an upper bound on  $\CL_{\tbpqp} (u_0, u_0z^N)$ for $N\in \im \zeta_{u_0}$  satisfying  \eqref{e:N bound}.

If $[\bar{u}] = [u_0]$ does not intersect $C_{\bpqp}(b)$ in $\bpqp$, then Proposition~\ref{p:zetas}\eqref{p:zetas2} renders this task straightforward: $\zeta_{u_0}$ is the zero map, and so $N=0$ and $\CL_{\tbpqp} (u_0, u_0z^N) =0$.  The case where $[u_0] = \{ 1 \}$ is also elementary and has $N=0$.

If, on the other hand, $[u_0]$ intersects $C_{\bpqp}(b) \ssm \{1 \}$  in $\bpqp$, then the following proposition applies with  $\gamma = u_0$.  (This is where we use the assumption, allowed by Proposition \ref{p:preferred new}, 
 that $u_0 \in C_{\bpqp}(b)$.)  Any lift $\tilde{g}$ of the element $g$  provided by Proposition \ref{p:translate}
  will conjugate $u_0$ to $u_0 z^N$ in $\tbpqp$.  If we choose a geodesic lift, then by combining 
  \eqref{e:K} and   \eqref{e:N bound}  we get
  $d_{\tbpqp}(1,g) = d_{\bpqp}(1,g) \leq E\, n^{\alpha +1}$ for a suitable constant $E$,
  where we have used the fact that  $\alpha >1$ to absorb the quadratic term from  \eqref{e:K}.

The resulting upper bound $\CL_{\tbpqp} (n)  \preceq n^{\alpha +1}$ matches our lower bound 
from Lemma~\ref{l: lower bound on CL} and therefore completes the proof of Theorem~\ref{t:main}.

\begin{proposition}\label{p:translate}
There exists $K \geq 1$ such that for all $\g \in C_{\bpqp}(b)$ and all $N \in \im \zeta_{\g}$,  there exists    $g \in C_{\bpqp}(\g)$ for which $N = \zeta_{\g}(g)$ and 
\begin{equation}\label{e:K}
d_{\bpqp}(1,g)\le K \, \left(|N| + d_{\bpqp}(1,\gamma)^2\right).
\end{equation}
\end{proposition}

\begin{proof}
By Proposition~\ref{p:zetas}, $C_{\bpqp}(b) = \<b\> \times F(a, \th)$. 

Suppose first that $\gamma = b^l$ for some $l \neq 0$. Proposition~\ref{p:zetas}\eqref{p:zetas3} tells us that  
in this case   
$\im \zeta_{\gamma}$ is generated by $l = \zeta_\g(\th )$. So $N\in l\Z$ and  for  $g=\th^{N/l}$ 
we have  $d_{\bpqp}(1,g) = |N/l| \leq |N|$. 

Suppose now that $\g \in C_{\bpqp}(b)  \ssm \langle b \rangle$.  So $\g = b^l \omega$ in $\bpqp$, where $l \in \Z$ and $\omega=\omega_0^m$ for some $\omega_0\in F(a,\th) \ssm \{1 \}$ that is not a proper power.  Let $j_0$ and $j$ be the exponent sums of $\th$ in $\omega_0$ and $\omega$, respectively.  So $j=mj_0$.  

Proposition~\ref{p:zetas}\eqref{p:zetas4} tells us that in this case the image of  $\zeta_{\g}$ is 
generated by 
$\zeta_{\g}(b) = j$ and $\zeta_{\g}(\omega_0) = j_0 l$.  So $N  = \lambda j + \mu j_0 l$ for some  $\lambda, \mu \in \Z$.  
By applying Lemma~\ref{l:bezout} with $b$ in the role of $x_0$ and $\omega_0$ in the role of $x_1$ and $\gamma$ in the role of $g$, we may assume that $\lambda, \mu \in \Z$ satisfy   
\begin{align}
 |\lambda| & < |N/j| + | j_0 l|, \text{ and} \label{e:lam} \\
|\mu| & \le |j|. \label{e:mu}
\end{align}
Define $g := b^\lambda \omega_0^\mu$ and note that $\zeta_{\g}(g) =N$.

We require two more estimates.  First,  Corollary~\ref{c: length in bpqp}  
gives us a constant $C>0$ such that
\begin{align}\label{e:big1}
\max \{ d_{\bpqp}(1,b^l), d_{\bpqp}(1,\omega_0) \} \leq C  d_{\bpqp}(1,\g). 
\end{align} 
Secondly,  because killing the other generators retracts $\tbpqp$ onto $\<\th\>$ and the
image of $\gamma$ under this retraction is $\th^j$,  we have
\begin{equation} \label{e:j} |j|  \le   d_{\bpqp}(1,\g). \end{equation}  
Now,  combining  \eqref{e:mu}--\eqref{e:j} we get
\begin{equation}\label{e:mu2}
d_{\bpqp}(1,\omega_0^\mu) \le |\mu|\, d_{\bpqp}(1,\omega_0) \le C  \, d_{\bpqp}(1,\gamma)^2.
\end{equation}

Using Proposition~\ref{p:distort} for the inequality and the retraction $\bpqp \onto \bpq$ killing $\th$ for the second equality, we also have 
\begin{equation}\label{e:l}
|l| = d_{\mathbb{T}}(1,b^l) \leq \k^{\alpha} d_{\bpq}(1,b^l)^{\alpha} = \k^{\alpha}  d_{\bpqp}(1,b^l)^{\alpha}. 
\end{equation} 
 Then, using (\ref{e:lam}) for the first inequality, and combing \eqref{e:big1}, \eqref{e:j}, and \eqref{e:l} for the third, we have
\begin{equation} \label{e:lambda}
	|\lambda| < |N/j| + |j_0 l| \le |N| + |j|.|l| \le  |N| +  \k^{\alpha} C^{\alpha}  d_{\bpqp}(1,\g)^{\alpha+1}.
\end{equation}

Using Proposition \ref{p:distort} and then \eqref{e:lambda}, we get  
\begin{equation}\label{e12}
d_{\bpqp}(1,b^\lambda) = d_{\bpq}(1,b^\lambda)  \leq  \k |\lambda|^{1/\alpha} \le  \k \left( |N| +  \k^{\alpha} C^{\alpha}  d_{\bpqp}(1,\g)^{\alpha+1} \right)^{1/\alpha}.
\end{equation} 
Finally, the triangle inequality applied to  $g = b^\lambda \omega_0^\mu$ gives
$$
d_{\bpqp}(1, g) \le d_{\bpqp}(1, b^\lambda) + d_{\bpqp}(1,\omega_0^\mu), $$
and then  \eqref{e:mu2} and \eqref{e12} yield \eqref{e:K}, as required, for suitable $K \geq 1$,  because $\alpha > 1$ 
and hence $(\alpha +1) / \alpha < 2$.  
 \end{proof}

\bibliographystyle{alpha}
\bibliography{bibli}

\bigskip

\noindent Martin R.\ Bridson, 
Mathematical Institute, Andrew Wiles Building, Oxford OX2~6GG, United Kingdom, \texttt{bridson@maths.ox.ac.uk},  
\href{https://people.maths.ox.ac.uk/~bridson/}{\texttt{people.maths.ox.ac.uk/bridson/}}

\bigskip

\noindent  {Timothy R.\ Riley}, Department of Mathematics, 310 Malott Hall,  Cornell University, Ithaca, NY 14853, USA,  \texttt{tim.riley@math.cornell.edu}, \href{https://pi.math.cornell.edu/~riley/index.html}{\texttt{math.cornell.edu/$\sim$riley/}}

\end{document}